\documentclass[12pt]{article}

\usepackage[margin=1in]{geometry}  
\usepackage{graphicx}              
\usepackage{amsmath}               
\usepackage{amsfonts}              
\usepackage{amsthm}                
\usepackage{color}
\usepackage{tabulary}
\usepackage{slashbox}

\newtheorem{thm}{Theorem}[section]
\newtheorem{lem}[thm]{Lemma}
\newtheorem{prop}[thm]{Proposition}
\newtheorem{cor}[thm]{Corollary}

\newtheorem{defn}[thm]{Definition}
\newtheorem{rmk}[thm]{Remark}
\newtheorem{ques}[thm]{Question}

\DeclareMathOperator{\id}{id}
\DeclareMathOperator{\Id}{Id}
\DeclareMathOperator{\Log}{Log}
\DeclareMathOperator{\Arg}{Arg}
\DeclareMathOperator{\PR}{PSL_2(\RR)}
\DeclareMathOperator{\PC}{PSL_2(\CC)}
\DeclareMathOperator{\ES}{\bf{E}}
\DeclareMathOperator{\HS}{\bf{H}}
\DeclareMathOperator{\PS}{\bf{P}}
\DeclareMathOperator{\CS}{\bf{C}}
\DeclareMathOperator{\EHS}{\bf{EH}}
\DeclareMathOperator{\WW}{\bf{W}}

\DeclareMathOperator{\CCC}{\mathcal{C}}
\DeclareMathOperator{\PPP}{\mathcal{P}}

\DeclareMathOperator{\EEE}{\mathcal{E}}
\DeclareMathOperator{\HHH}{\mathcal{H}}

\DeclareMathOperator{\FF}{\mathcal{F}}
\DeclareMathOperator{\Aut}{Aut}
\DeclareMathOperator{\CP1}{\mathbb{C}\mathbb{P}^1}      %
\DeclareMathOperator{\CPP}{\mathbb{C}\mathbb{P}^2}      %
\DeclareMathOperator{\SO}{SO_4}      %
\DeclareMathOperator{\SOO}{SO_2}      %
\DeclareMathOperator{\Dp}{\mathbb{D}^{+}}      %
\DeclareMathOperator{\Dm}{\mathbb{D}^{--}}      %
\DeclareMathOperator{\Chat}{\widehat{\mathbb{C}}}
\DeclareMathOperator{\lcm}{lcm}
\DeclareMathOperator{\Fix}{Fix}
\DeclareMathOperator{\Max}{Max}
\DeclareMathOperator{\sign}{sign}
\DeclareMathOperator{\RE}{Re}
\DeclareMathOperator{\IM}{Im}

\newcommand{\NN}{\mathbb{N}}      %
\newcommand{\ZZ}{\mathbb{Z}}      
\newcommand{\RR}{\mathbb{R}}      
\newcommand{\QQ}{\mathbb{Q}}      
\newcommand{\CC}{\mathbb{C}}      
\newcommand{\DD}{\mathbb{D}}      
\newcommand{\HH}{\mathbb{H}}      

\begin{document}


\title{The Space of Geometric Limits of Abelian Subgroups of $\PC$}

\author{Hyungryul Baik \& Lucien Clavier \thanks{We really appreciate that John H. Hubbard let us know about this problem and explained how we could approach at the beginning. He also has provided us a lot of advices through enlightening discussions. We also thank to Bill Thurston for the helpful discusstions. } \\ 
Department of Mathematics\\
310 Malott Hall, Cornell University \\
Ithaca, New York 14853-4201 USA}

\maketitle

\begin{abstract}
We describe the topology of the space of all geometric limits of closed abelian subgroups of $\PC$. 
Main tools and ideas will come from the previous paper \cite{BC1}. 
\end{abstract}

\tableofcontents
\listoffigures

\section{Introduction}
\label{intro}
\subsection{Chabauty topology}

Recall that the Chabauty topology of a locally compact group $G$ is the topology on the space $\overline{F}(G)$ of all its closed subgroups induced by the Hausdorff distance one the one-point compactification $\overline{G}$ of $G$ (see \cite{BC1} for instance, or \cite{Cha1}).
Equipped with this topology, $\overline{F}(G)$ becomes a compact metric space;
$\overline{F}(G)$, together with the Hausdorff distance $d_H$, will be usually referred to as the \textit{Chabauty space} of $G$. We write it $\CCC(G)$.

In the context of Kleinian groups, the limit of a convergent sequence in the Chabauty topology is called the \emph{geometric limit} of the sequence.

 In the previous paper \cite{BC1} of the authors, we obtained the following theorem, where $\CS$ (resp. $\ES$, $\HS$, $\PS$) is the closure of the space of discrete cyclic subgroups of $\PR$ (resp. discrete subgroups generated by an elliptic, hyperbolic, parabolic element of $\PR$). 
 
\begin{thm}
 \label{thm:mainthmofpsl2r} 
 The space of all geometric limits of closed subgroups of $\PR$ with one generator is $\CS = \ES \cup \HS / \sim$, where
\begin{itemize}
 \item[(1)] $\ES$ is a wedge sum of countably many 2-spheres $D_n / \partial D_n $, which accumulate on a disk $D_\infty$ and to the cone $\PS$ on the circle $\partial D_\infty$.
(see Figure 7 in \cite{BC1})
 \item[(2)] $\HS$ is the cone on a closed M\"{o}bius band, the inside of which is foliated by ``bent'' open M\"{o}bius bands, which accumulate to an open M\"{o}bius band $M_0$ and the cone $\PS$ on the circle $\partial M_0$ (see Figure 8 in \cite{BC1}).
 \item[(3)] $\sim$ represents the gluing of $\ES$ and $\HS$ along $\PS$.
\end{itemize}
\end{thm} 

The geometric convergence of Kleinian groups is not so easy to understand in general; we developed the following proposition to reduce the problem of the convergence for the Hausdorff topology in some complicated space (e.g. $\CS\subset \PR$) to the convergence for the Hausdorff topology in a better-known space (e.g. some particular family of closed subsets of $\CC$). 
This will play a fundamental role in the present paper.

\begin{prop}[Reduction Lemma]
\label{prop:redlem}
Let $(X,d_X),\, (Y,d_Y)$ be two second countable, locally compact metric spaces. 
Let $( \varphi_n)$ be a sequence of maps from $X$ to $Y$, converging to a continuous proper map~$\varphi$, uniformly on every compact subset.
Assume that for every compact subset $K \subset Y$, the closed subset 
\[
\overline{\bigcup_{n\geq N}\varphi_n^{-1}(K)}
\]
is compact for $N$ large enough.

Then whenever a sequence of closed subsets $F_n\subset X$ converges to a closed subset $F$ in the Hausdorff topology of $X$, the subsets $\overline{\varphi_n(F_n)}$ converge to $\overline{\varphi(F)}$ in the Hausdorff topology of $Y$.
\end{prop}

\subsection{Transformations of $\PC$}
Note, after identification of $\PC$ and $\Aut(\HH^3)$, that each element of $\PC$ acts on $\Chat=\partial\HH^3$.
Let us recall that the isometries of $\HH^3$ are of three types:
\begin{itemize}
\item parabolic if they have one fixed point in $\Chat=\partial\HH^3$.
\item elliptic if they have two fixed point in $\Chat$, and act on $\HH^3$ as a rotation along the axis defined by these fixed points.
\item hyperbolic if they have two fixed points in $\Chat$, and act on $\HH^3$ as a translation with skew along the axis defined by these fixed points.
\end{itemize}
More precisely, elliptic (resp. hyperbolic) elements of $\PC$ are conjugated to a map $[z\mapsto a z]$, by sending one of the fixed point at 0, and the other at $\infty$; we have $a\in S^1$ (resp. $a\in \CC^\ast\setminus S^1$), and we call $a$ the rotation (resp. translation) quantity of the element. 


\subsection{Closed abelian subgroups of $\PC$}

In this paper, we are mainly interested in subgroups of $\PC$ that are cyclic or abelian.

For every isometry $g\in \Aut(\HH^3)$, let $\Fix(g)$ be the set of fixed points of $g$ on the sphere at infinity $\Chat$. 
Recall that abelian subgroups of $\PC$ are exactly the subgroups $G\subset \PC$ such that every elements have the same fixed points at infinity, i.e.
\[
\forall g_1,g_2\in G, \, \Fix(g_1)=\Fix(g_2)
\]

Now, let us define $\CS_1$ (resp. $\CS_2$) to be the boundary of the space of cyclic (resp. abelian) subgroups of $\PC$.
Of course, $\CS_1\subset \CS_2$, but $\CS_1$ is an interesting object in itself.

Note that a non-trivial abelian subgroup $G$ of $\PC$ can be of two rather different kinds, namely:
\begin{itemize}
\item $G$ is a parabolic subgroup, i.e. each  of its non-trivial element is parabolic and fixes the same point $z\in \Chat$.
Then $G$ is conjugated to a subgroup $\Gamma$ of translations of $\Chat$, under a map sending $z$ to $\infty$;
if this map is chosen once and for all, $G$ is entirely described by $z$ and $\Gamma$ (see Section \ref{parabolics})
\item $G$ is a non-parabolic subgroup, i.e. each of its non-trivial element is non-parabolic and fixes the same points $z_1, z_2\in \Chat$.
Then the space of all rotation/translation quantities of its elements is a subgroup $\Xi$ of $\CC^\ast$, and $G$ is entirely described by the unordered pair $\{z_1,z_2\}$ and by $\Xi$. 
\end{itemize}

Moreover, $\CS_1$ is constituted of three main ``chunks'', namely $\PS_1$ (resp. $\ES$, $\HS$) the boundary of the space of cyclic subgroups generated by one parabolic (resp. elliptic, hyperbolic) element.
Other special ``chunks'' of $\CS_2$ are $\PS_2$ the boundary of the space of abelian parabolic subgroups, and for each $m\ge 1$ the spaces $\EHS_m$, respectively the boundaries of the spaces of non-parabolic subgroups of $\PC$ which have exactly $m$ elliptic elements (counting the identity of $\PC$ as an elliptic element here).
 In particular, $\EHS_1=\HS$.

Exactly how these ``chunks'' fit together is the purpose of the paper.

See Section \ref{sec:future} for a summary statement.

\subsection{Gaining a generator}

To finish this introductory part, let us expose the following example of a sequence of cyclic hyperbolic subgroups converging to a non-cyclic parabolic subgroup.
The existence of such a behaviour is well-known (see for instance \cite{Brock}); the following explicit example is due to John H Hubbard, and can be found in the (unpublished) book \cite{Hubb2}.

Consider a sequence of elements $\alpha_n$ of $\PC$ of the form 
\[
\alpha_n : z \mapsto \rho_n^2 e^{\frac{2\pi i }{n}} (z - a_n) + a_n.
\]
 Each $\alpha_n$ is a hyperbolic isometry which fixes both $a_n$ and $\infty$.
 We will choose both $\rho_n$ and $a_n$ so that $\alpha_n$ converges to $[z \mapsto z+1]$ 
and $\alpha_n^n$ converges to the parabolic element $[z \mapsto z +
\gamma]$ with $\gamma \notin \RR$.

For that purpose, set 
\[
a_n = \dfrac{1}{1 - \rho_n^2 e^{\frac{2\pi i }{n}}}
 \]
so that $\alpha_n(0)\equiv 1$.
In particular, this forces $\rho_n\to 1$.

When $\rho_n \equiv 1$, $\alpha_n^n = \id$ for all $n$, so it obviously converges to the identity.
When $\alpha_n^n = 1 + \frac{\alpha}{n}$ for some constant $\alpha$, it is not convergent.
We leave to the reader to check that if $\rho_n \in o(\frac{1}{n^2})$, then $\alpha_n^n$ 
converges to a parabolic element, and that if $\rho_n$ is defined by
\[
 \rho_n = 1 +\dfrac{\alpha}{n^2}
\]
then $\alpha_n^n$ converges to the map $[z \mapsto z -i\frac{\alpha}{\pi}] $.

There is a nice geometric view of this enrichment, due to Jorgersen, using an invariant cone becoming wider and wider (see \cite{Brock}).
In Section \ref{subsec:cylinders}, we will give another way of seeing this phenomenon, which is similar but more intimately connected to the present work.

\pagebreak

\section{Matrix representations}
\subsection{$\CP1$ as a quotient}
\label{cp1}
For technical reasons that will appear clearly in the next two sections, we would like to find a particular subspace of $\CC^2\setminus \{0\}$ mapped homeomorphically to $\CC$ via $(\zeta,\xi)\mapsto \zeta/\xi$, and which stays away from both 0 and $\infty$ (i.e. has a compact closure that does not contains 0).
The first classical choice of such a homeomorphism is the plane $\CC \times \{1\}$, but it is not compact; the second is the sphere in $\CC \times \RR \subset \CC^2$ of radius $1/2$ and centered at (0,1/2), with the south pole removed:
\[
S^2\setminus S= \{(\zeta,\xi)\in\CC^2\setminus \{0\};\; \xi\in (0,1],\, |\zeta|^2 + (\xi-1/2)^2 = 1/4 \}
\]
but its closure $S^2$ contains the south pole $S=(0,0)$.

Let us then define $\Dp$ to be the unit upper hemisphere in $\CC \times \RR \subset \CC^2$:
\[
 \Dp=\{(\zeta,\xi)\in\CC^2\setminus \{0\};\; \xi\in (0,1],\, |\zeta|^2 + \xi^2 = 1 \}
\]

\begin{figure}[ht]
\begin{center}
\includegraphics[scale=0.3]{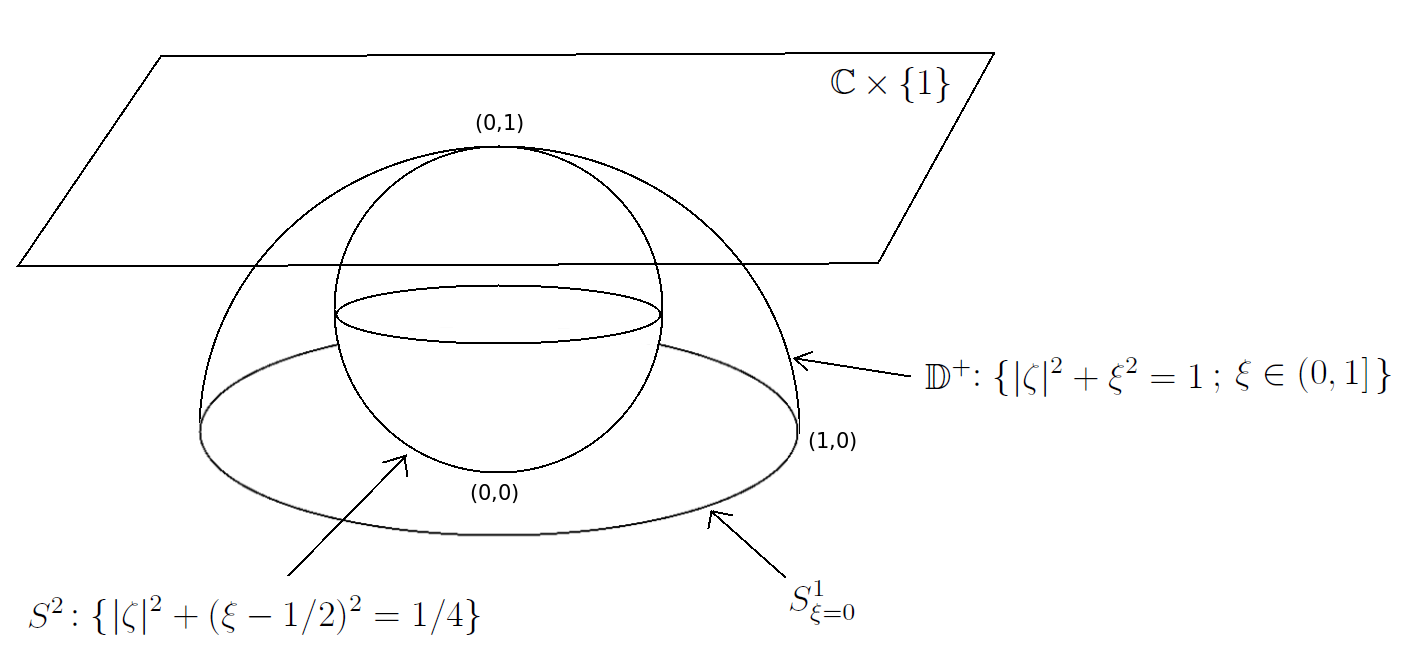}
\end{center}
\caption[Three models for $\CP1$ with a point removed: the plane $\CC\times \{1\}$]{Three models for $\CP1$ with a point removed: the plane $\CC\times \{1\}$, the sphere $S^2$ minus its south pole $(0,0)$, and the upper unit hemisphere $\Dp$.}
\label{fig:cp1}
\end{figure}  

Note that $(\zeta,\xi)\mapsto [\zeta:\xi]$ induces an homeomorphism
\[ \overline{\Dp}/ S^1_{\xi=0} \cong \CP1\]

\begin{rmk}
\label{rmk:stereo}
It is straightforward to see that the stereographic projection from $\overline{\Dp}$ to $S^2$ is given by 
\[
(\zeta,\xi)\mapsto (\zeta\xi,\xi^2)
\]
\end{rmk}

Let us define
\[
\Dm=(\overline{\Dp}/ S^1_{\xi=0})\setminus \{(0,1)\}
\]
and 
\[
S^1_{\text{eq}}=\{(\zeta,\xi)\in \Dp;\; \xi=\frac{\sqrt{2}}{2}\}
\]

Under the stereographic projection, $\Dp$, $\Dm$ and $S^1_{\text{eq}}$ are respectively mapped to $S^2\setminus S$, $S^2\setminus N$ (where $S$ and $N$ are the south and north poles $(0,0)$ and $(0,1)$) and to the equator of $S^2$ for which $\xi$ is constantly 1/2.

\subsection{Matrix representations of elliptic and hyperbolic isometries}
\label{matrixEH}
In this section, we show how to represent every elliptic and hyperbolic element of $\PC$ as a $2 \times 2$ matrix.

Recall that an elliptic (resp. hyperbolic) element of $\PC$ fixes an unique axis joining two distinct fixed points of $\Chat$, and that it then acts like a rotation (resp. a complex translation) around this axis.
Moreover, an elliptic (resp. hyperbolic) element is entirely determined by its two fixed points and its rotational angle (resp. complex translation number). 

Let $H$ be either elliptic or hyperbolic, with rotation/translation quantity $a\in \CC^\ast$; suppose first that the two distinct fixed points $z_1$ and $z_2$ of $H$ are in $\CC$. 
Then it is straightforward to check that the matrix of $\PC$ representing $H$ is
\[
 \frac{1}{\sqrt{a}(z_2-z_1)} \begin{pmatrix} az_2-z_1 & z_1z_2(a-1) \\ 1-a & z_2-az_1 \end{pmatrix}
\]
Indeed, $\phi = [z \mapsto \frac{z - z_1}{z - z_2}]$ is an automorphism of $\Chat$ mapping $z_1$ to $0$ and $z_2$ to $\infty$;
moreover, the element with rotation/translation quantity $a$ and fixed points $0$ and $ \infty$ is simply $[z \mapsto a z]$.

This description using $z_i\in \CC$ has the drawback of blowing off when one of these fixed points approaches $\infty\in \Chat$.
To circumvent this, let us replace $z_1$ and $z_2$ by projectivized quantities $\zeta_1/\xi_1$ and $\zeta_2/\xi_2$, with $(\zeta_i,\xi_i)\in \Dp$. 
Then the matrix becomes
\[
 \frac{1}{\sqrt{a}} \begin{pmatrix} 1+\mu\,\zeta_2\xi_1 & \mu\,\zeta_1\zeta_2 \\-\mu\,\xi_1\xi_2  & 1-\mu\,\zeta_1\xi_2 \end{pmatrix}
\]
with $\mu=\dfrac{a-1}{\zeta_2\xi_1-\zeta_1\xi_2}$.

\begin{defn}
\label{defn:matrepEH}
For every pair of distinct points $(\zeta_1,\xi_1),(\zeta_2,\xi_2)\in \Dp$, and for every $a\in \CC$ verifying $|a|=1$ (resp. $|a|\notin \{0,1\}$), the elliptic (resp. hyperbolic) element of $\PC$ fixing both $[\zeta_i:\xi_i]\in \CP1$ and with rotation (resp. translation) quantity $a$ is written 
\[
 H_{(\zeta_1,\xi_1),(\zeta_2,\xi_2),a}= \frac{1}{\sqrt{a}} \begin{pmatrix} 1+\mu\,\zeta_2\xi_1 & \mu\,\zeta_1\zeta_2 \\-\mu\,\xi_1\xi_2  & 1-\mu\,\zeta_1\xi_2 \end{pmatrix}
\]
with $\mu=\frac{a-1}{\zeta_2\xi_1-\zeta_1\xi_2}$.
\end{defn}

Write $\Theta$ for the space of pairs of distinct points in $S^2$:
\[
\Theta=\Big((\CP1\times\CP1)\setminus \Delta\Big)\Bigg/ \Big((x,y)\sim(y,x) \Big) 
\]
where $\Delta$ is the diagonal $\{(x,x);\; x\in \CP1\}$; we definitely think of $\CP1$ here as being $\overline{\Dp}/ S^1_{q=0} $.

\begin{lem}
The space $\EEE$ (resp. $\HHH$) of non-trivial elliptic (resp. hyperbolic) elements of $\PC$ is homeomorphic to $\Theta\times (S^1\setminus 1)$ (resp. $\Theta\times (\CC^\ast \setminus S^1)$).
\end{lem}
\begin{proof}
 With the choice of $\Dp$ for representing $\CP1\setminus \{[1:0]\}$,
$H_{(\zeta_1,\xi_1),(\zeta_2,\xi_2),a}$ can be extended continuously when either one of the $(\zeta_i,\xi_i)$ approaches the boundary of $\Dp$. 
We write this continuation in the obvious way, e.g. if $\xi_1=0$ (thus $|\zeta_1|=1$):
\[
 H_{(\zeta_1,0),(\zeta_2,\xi_2),a}= \frac{1}{\sqrt{a}} \begin{pmatrix} 1 & \mu\,\zeta_1\zeta_2 \\ 0  & 1-\mu\,\zeta_1\xi_2 \end{pmatrix}
\]

Note that, because of the definition of $\mu=\frac{a-1}{-\zeta_1\xi_2}$, this matrix does not depend on $\zeta_1$.
Therefore, the desired homeomorphism for the elliptic case is
\begin{align*}
\Theta \times (S^1\setminus 1)& \to \PC\\
(\zeta_1,\xi_1),(\zeta_2,\xi_2),e^{i\phi} &\mapsto  H_{(\zeta_1,\xi_1),(\zeta_2,\xi_2),e^{i\phi}}
\end{align*}

Similarly, the desired homeomorphism for the hyperbolic case is
\begin{align*}
\Theta \times ( \CC^\ast\setminus S^1)& \to \PC\\
(\zeta_1,\xi_1),(\zeta_2,\xi_2),a &\mapsto  H_{(\zeta_1,\xi_1),(\zeta_2,\xi_2),a}
\end{align*}
\end{proof}

This description have the following two propositions as a consequence.

\begin{prop}
\label{prop:nonpar}
The space of all non-trivial non-parabolic closed abelian subgroups of $\PC$ is homeomorphic to 
\[
\Theta \times (\CCC(\CC^\ast)\setminus \{1\})
\] 
\end{prop}
\begin{proof}
It is immediate to see that the desired homeomorphism is induced by
\begin{align*}
\Theta \times (\CCC(\CC^\ast)\setminus \{1\})\to& \CCC(\PC) \\
(\zeta_1,\xi_1),(\zeta_2,\xi_2),\Xi \mapsto& \left\{ H_{(\zeta_1,\xi_1),(\zeta_2,\xi_2),a};\; a\in \Xi \right\} \\
\end{align*}
where by convention $H_{(\zeta_1,\xi_1),(\zeta_2,\xi_2),1}$ is always the identity of $\PC$.
More precisely, the map above is a homeomorphism from $\Theta \times (\CCC(\CC^\ast)\setminus \{1\})$ onto its image, which is the space of all non-trivial non-parabolic closed abelian subgroups of $\PC$, and we are done.
\end{proof}

\begin{prop}
\label{prop:nonpar2}
The space of all non-trivial discrete subgroups of $\PC$ generated by one elliptic generator is homeomorphic to 
\[
\Theta \times \NN_{\ge 2}
\] 

The space of non-trivial discrete subgroups of $\PC$ generated by one hyperbolic generator is homeomorphic to 
\[
\Theta \times (\CC\setminus \overline{\DD})
\] 

\end{prop}
\begin{proof}
It is immediate to see that the desired homeomorphism for the elliptic case is induced by
\begin{align*}
\Theta \times \NN_{\ge 2}& \to \CCC(\PC)\\
(\zeta_1,\xi_1),(\zeta_2,\xi_2),n &\mapsto \langle H_{(\zeta_1,\xi_1),(\zeta_2,\xi_2),e^{2\pi i/n}}\rangle
\end{align*}
where as usual the chevrons indicate ``subgroup generated by''.

Similarly, the desired homeomorphism for the hyperbolic case is induced by
\begin{align*}
\Theta \times ( \CC\setminus \overline{\DD})& \to \CCC(\PC)\\
(\zeta_1,\xi_1),(\zeta_2,\xi_2),a &\mapsto \langle H_{(\zeta_1,\xi_1),(\zeta_2,\xi_2),a}\rangle
\end{align*}
\end{proof}

Before going any further with parabolic subgroups, let us give a more recognizable form to $\Theta$.
\subsection{$\Theta$ as a subspace of $\CPP$}
\label{subsec:Theta}
This subsection is due to John Hamal Hubbard, and we thank him to explain it to us.

For each pair $([\zeta_1:\xi_1],[\zeta_2:\xi_2])\in (\CP1)^2$, consider the polynomial 
\[
P_{[\zeta_1:\xi_1],[\zeta_2:\xi_2]}(x)=(\xi_1x-\zeta_1)(\xi_2x-\zeta_2) 
\]
defined up to a multiplicative constant.

Note that $P_{[\zeta_1:\xi_1],[\zeta_2:\xi_2]}$ and $P_{[\zeta_3:\xi_3],[\zeta_4:\xi_4]}$ differ by a multiplicative constant if and only if $\{[\zeta_1:\xi_1],[\zeta_2:\xi_2]\}=\{[\zeta_3:\xi_3],[\zeta_4:\xi_4]\}$ as sets. 
Also, $[\zeta_1:\xi_1]=[\zeta_2:\xi_2]$ if and only if the square-rooted discriminant of $P_{[\zeta_1:\xi_1],[\zeta_2:\xi_2]}$
\[
\sqrt{\Delta_P}=\pm (\zeta_2\xi_1-\zeta_1\xi_2)
\]
is zero (remark that this number is, up to sign, the denominator of the quantity $\mu$ defined in the matrix representations above, Definition \ref{defn:matrepEH}).

Therefore, the map 
\begin{align*}
\CP1\times\CP1& \to \CPP\\
[\zeta_1:\xi_1],[\zeta_2:\xi_2] &\mapsto [\xi_1\xi_2:-\zeta_2\xi_1-\zeta_1\xi_2:\zeta_1\zeta_2]
\end{align*}
descends to a homeomorphism between $\Theta$ and $\CPP$ minus the curve (actually a sphere) of homogeneous equation
\[
Y^2-4XZ=0
\]

\subsection{Matrix representations of parabolic isometries}
\label{parabolics}
In this section, we show how to represent every parabolic element of $\PC$ as a $2 \times 2$ matrix. This will lead to a complete description of $\PS_1$ and $\PS_2$.

It is well-known that every two parabolic elements of $\PC$ that fix the same point of $S^2$ are conjugate.
We saw in \cite{BC1} that we could find a way of normalizing parabolic elements, by asking them to be conjugated to a particular parabolic element by a chosen matrix. 
This led to a homeomorphism between the space of non-trivial parabolic elements of $\PR$ and $S^1\times  \RR^\ast$, the upshot being a description of the space of parabolic cyclic subgroups of $\PR$.

We cannot immediately extend this method to the case of $\PC$; we will actually see in a while that the space of non-trivial parabolic elements of $\PC$ is a non-trivial bundle (see Proposition \ref{prop:ppp}). 

Write $\PPP$ for the space of all non-trivial parabolic elements of $\PC$.
Since for each point $z\in S^2$ the space of parabolic elements of $\PC$ fixing $z$ is homeomorphic to $\CC^\ast$,we see that $\PPP$ is a $\SOO$-bundle, with base space $S^2$ and fiber $\CC^\ast$.

Now we have two local trivializations
\begin{align*}
\Dp\times \CC^\ast & \to \PPP \subset \PC\\
(\zeta,\xi),\rho &\mapsto \begin{pmatrix} 1 - \rho \, \zeta\xi &  \rho \, \zeta^2  \\ - \rho \, \xi^2  & 1 +  \rho \, \zeta\xi \end{pmatrix} 
\end{align*}
and
\begin{align*}
\Dm\times \CC^\ast & \to \PPP \subset \PC \\
(\zeta,\xi),\rho &\mapsto \begin{pmatrix} 1 - \rho \, \overline{\zeta}\xi &  \rho  \, |\zeta|^2 \\ - \rho \, \dfrac{\overline{\zeta}}{\zeta} \xi^2  & 1 +  \rho \, \overline{\zeta}\xi \end{pmatrix} 
\end{align*}

The clutching map associated with these two trivializations is 
\begin{align*}
S^1_{\text{eq}}& \to \SOO\\
e^{i\phi} &\mapsto [\rho \mapsto e^{2i\phi} \rho]
\end{align*}
which represents the number 2 in $H_1(\SOO)\cong \ZZ$. Thus we have proven the following:
\begin{prop}
\label{prop:ppp}
The space $\PPP$ of non-trivial parabolic elements of $\PC$ is a 2-twist fiber bundle of $\CC^\ast$ above $S^2$. 
\end{prop}

\subsection{The spaces $\PS_1$ and $\PS_2$}
\label{subsec:parabolic}
Define $\PS_1'$ (resp. $\PS_2'$) to be the space of all non-trivial discrete cyclic (resp. abelian) parabolic subgroups of $\PC$; of course $\PS_i'\subset\PS_i$.
As above, we see that  $\PS_i'$ is a fiber-bundle over $S^2$ of the space of all non-trivial discrete cyclic subgroups (resp. non-trivial discrete subgroups) of $\CC$.
The former fiber is easily seen to be simply $\CC^\ast \Big/ (z\sim-z) \cong \CC^\ast$, hence $\SOO$ is the structure group of the bundle $P_1'$; 
the latter is known from \cite{PourHubb} to be homeomorphic to $(\CC^2)^\ast$, hence $\SO$ is the structure group of $P_2'$,
where $\SO$ acts on $ (\CC^2)^\ast$ in the usual way (namely, as a $4 \times 4$ real-matrix group acts on $ (\CC^2)^\ast = \RR^4\setminus 0$).

\begin{prop}
\label{prop:p1}
$\PS_1'$ is a 4-twist fiber bundle of $\CC^\ast$ above $S^2$
\end{prop}
\begin{proof}

As above we have two trivializations
\begin{align*}
\Dp\times \CC^\ast \Big/( z\sim-z) & \to \,\PS_1' \\
(\zeta,\xi),u &\mapsto 
\left\{\begin{pmatrix} 1 - \rho \, \zeta\xi &  \rho \, \zeta^2  \\ - \rho \, \xi^2  & 1 +  \rho \, \zeta\xi \end{pmatrix}  ;\; \rho \in \langle u \rangle \right\}
\end{align*}
and
\begin{align*}
\Dm\times \CC^\ast \Big/ (z\sim-z )& \to \, \PS_1' \\
(\zeta,\xi),u &\mapsto 
\left\{\begin{pmatrix} 1 - \rho \, \overline{\zeta}\xi &  \rho \, |\zeta|^2  \\ - \rho \, \dfrac{\overline{\zeta}}{\zeta} \xi^2  & 1 +  \rho \, \overline{\zeta}\xi \end{pmatrix} ;\; \rho \in \langle u \rangle\right\}
\end{align*}
where $\langle u \rangle$ is here the additive subgroup of $\CC$ generated by $\pm u\in \CC^\ast$.

The clutching map associated with these two trivializations is now
\begin{align*}
e^{i\phi} &\mapsto [\langle u \rangle \mapsto e^{2i\phi} \langle u \rangle]
\end{align*}
which becomes, after identifying $\CC^\ast \Big/ (z\sim-z)$ and $\CC^\ast$:
\begin{align*}
S^1_{\text{eq}}& \to \SOO\cong S^1\\
e^{i\phi} &\mapsto e^{4i\phi}
\end{align*}
and we are done.
\end{proof}

\begin{cor}
\label{cor:p1}
$\PS_1$ is the one-point compactification of a 4-twist $\SOO$-bundle of $\overline{\DD}^\ast$ over $S^2$.
\end{cor}
\begin{proof}
The closure of the space of discrete subgroups of $\CC$ for the Chabauty topology of $\CC$ is just a closed disc (see for instance \cite{PourHubb}). Thus this corollary follows from \ref{prop:p1}.  
\end{proof}

\begin{prop}
\label{prop:p2}
$\PS_2'$ is homeomorphic to $S^2\times (\CC^2)^\ast$.
\end{prop}
\begin{proof}
The space $\CCC_d(\CC)$ of all discrete subgroups of $\CC$ is homeomorphic to $\CC^2$ via a map $F$ explicited in Section 3 of \cite{PourHubb}. 
This map $F$ is the inverse of 
\[
\Gamma \mapsto \Big( \dfrac{1}{60}\displaystyle\sum_{z\in \Gamma\setminus 0} \dfrac{1}{z^4},\dfrac{1}{140}\displaystyle\sum_{z\in \Gamma\setminus 0} \dfrac{1}{z^6}\Big)
\]
One recognizes at once the use of the functions $G_2$ and $G_3$, classical in the theory of elliptic functions.

Now as above we have two trivializations
\begin{align*}
\Dp\times \CCC_d(\CC) & \to \,\PS_2'\\
(\zeta,\xi),\Gamma &\mapsto 
\left\{\begin{pmatrix} 1 - \rho \, \zeta\xi &  \rho \, \zeta^2  \\ - \rho \, \xi^2  & 1 +  \rho \, \zeta\xi \end{pmatrix}  ;\; \rho \in \Gamma\right\}
\end{align*}
and
\begin{align*}
\Dm\times \CCC_d(\CC) & \to \,\PS_2'\\
(\zeta,\xi),\Gamma &\mapsto 
\left\{\begin{pmatrix} 1 - \rho \, \overline{\zeta}\xi &  \rho \, |\zeta|^2  \\ - \rho \, \dfrac{\overline{\zeta}}{\zeta} \xi^2  & 1 +  \rho \, \overline{\zeta}\xi \end{pmatrix} ;\; \rho \in \Gamma\right\}
\end{align*}

The clutching map associated with these two trivializations is now
\begin{align*}
e^{i\phi} &\mapsto [\Gamma \mapsto e^{2i\phi} \Gamma]
\end{align*}
which becomes, using $F$:
\begin{align*}
S^1_{\text{eq}}& \to \SO\\
e^{i\phi} &\mapsto [(a,b) \mapsto (e^{8i\phi}a,e^{12i\phi}b)]
\end{align*}
(see Lemma 2 in \cite{PourHubb}).
Since $\SO$ admits $S^3\times S^3$ (which is simply connected) as double cover, and since the clutching map above is the double of some loop in $\SO$, we see that it can be homotoped to the trivial loop 
\[
e^{i\phi} \mapsto [(a,b) \mapsto (a,b)]= \Id_{\SO}
\]
therefore $\PS_2'$ is homeomorphic to  $S^2\times\CCC_d(\CC)$.
\end{proof}

\begin{cor}
\label{cor:p2}
$\PS_2$ is homeomorphic to the one-point compactification of  $S^2\times\RR^4$.
\end{cor}
\begin{proof}
We can conclude, simply by considering the possible limits of elements of $\PS_2\setminus \{\Id\}$, that the homeomorphism described at the end of the proof of Proposition \ref{prop:p2} extends to a homeomorphism between $\PS_2\setminus \{\Id\}$ and $S^2\times (\CCC(\CC)\setminus \{0\})$; thus, since $\CCC(\CC)$ is homeomorphic to $S^4$ (see \cite{PourHubb}), we are done.
\end{proof}

\pagebreak

\section{Reduction lemma}
\label{sec:redlem}
\subsection{The two reducing arguments}
From Proposition \ref{prop:nonpar}, we know that any non-trivial non-parabolic closed abelian subgroup of $\PC$ is well defined by two fixed point and a closed subgroup of $\CC^\ast$.
Thus, let us consider a sequence $(G_n)$ of non-trivial non-parabolic closed abelian subgroups of $\PC$.
For all $n$, let us define $((\zeta_1)_n,(\xi_1)_n),((\zeta_2)_n,(\xi_2)_n)\in \overline{\Dp}$ such that $[(\zeta_1)_n:(\xi_1)_n],[(\zeta_2)_n:(\xi_2)_n]\in \CP1$ are the distinct fixed points of one of the generators of $G_n$ (the order does not matter);
let us also define $\Xi_n$ to be the subgroup of $\CC^\ast$ consisting of the rotation/translation quantities of the elements of $G_n$ (see Proposition \ref{prop:nonpar}).
For notational purposes, let us finally define $R_n\ge1$, $\omega_n \in [0,2\pi]$ such that 
\[
R_n e^{i\omega_n}=\dfrac{1}{(\zeta_2)_n(\xi_1)_n-(\zeta_1)_n(\xi_2)_n}
\]

 Taking extractions if necessary, we can always assume that $((\zeta_1)_n,(\xi_1)_n)$ and $((\zeta_2)_n,(\xi_2)_n)$ converge in $\overline{\Dp}$, that $(R_n)$ converges in $[1,\infty]$ and $(e^{i\omega_n}) $ converges in $S^1$.
We denote the limits of these quantities with the subscript $\infty$.

Let us now recall the Reduction Lemma, a key tool in the reducing arguments below.

\begin{prop}[Reduction Lemma]
\label{redlem}
Let $(X,d_X),\, (Y,d_Y)$ be two second countable, locally compact metric spaces. 
Let $( \varphi_n)$ be a sequence of maps from $X$ to $Y$, converging to a continuous proper map~$\varphi$, uniformly on every compact subset.
Assume that for every compact subset $K \subset Y$, the closed subset 
\[
\overline{\bigcup_{n\geq N}\varphi_n^{-1}(K)}
\]
is compact for $N$ large enough.

Then whenever a sequence of closed subsets $F_n\subset X$ converges to a closed subset $F$ in the Hausdorff topology of $X$, the subsets $\overline{\varphi_n(F_n)}$ converge to $\overline{\varphi(F)}$ in the Hausdorff topology of $Y$.
\end{prop}
\begin{proof}
See Section 5 in \cite{BC1}.
\end{proof}

\begin{rmk}
\label{rmk:omega}
If the functions $\varphi_n$ are only defined on some domains $\Omega_n \subset X$ verifying that for any compact subset $K\subset X$, we can find an integer $N$ such that for all $n\geq N$, $K\subset \Omega_n$ (or, equivalently, if for every neighborhood $\mathcal{N}$ of the infinity-point $\infty\in \overline{X}$ and for all $n$ large enough, $\Omega_n^c\subset \mathcal{N}$), then the conclusion of Proposition \ref{redlem} still holds \textit{if $F_n\subset \Omega_n$ for every $n$}, simply by declaring that $\widetilde{\varphi_n}$ sends every point of $\Omega_n$ to $\infty\in \widetilde{Y}$.
\end{rmk}

Theorems \ref{thm:reducingargument1} and \ref{thm:reducingargument2} below are the reducing arguments needed to reduce the problem of the convergence of non-parabolic groups in $\CS_2$ to problems about convergence in $\CCC(CC)$.
The former deals with the case where the geometric limit is parabolic, the latter deals with the easier case where the geometric limit is non-parabolic.

\begin{thm}
\label{thm:reducingargument1}
Let $G_n$ as above.

Suppose that $R_\infty=\infty$.
For all $n$, define $\Gamma_n$ to be the closed subgroup of $\CC$ defined by $\Gamma_n=R_n \Log(\Xi_n)$.
Then assuming that $\Gamma_n$ converges to some closed subgroup $\Gamma_\infty$ of $\CC$ for the Chabauty topology of $\CC$ (this can always be performed by extracting subsequences), the geometric limit of $G_n$ is the subgroup
 \[
  G_\infty=\left\{ \begin{pmatrix} 1 + \rho e^{i \omega_\infty} \zeta_\infty \xi_\infty &  \rho e^{i \omega_\infty} \zeta_\infty^2 \\
   - \rho e^{i \omega_\infty} \xi_\infty^2 & 1- \rho e^{i \omega_\infty} \zeta_\infty \xi_\infty \end{pmatrix};\; \rho \in \Gamma_\infty\right\}
 \]
 where we defined $\zeta_\infty$ to be $(\zeta_1)_\infty=(\zeta_2)_\infty$ and  $\xi_\infty$ to be $(\xi_1)_\infty=(\xi_2)_\infty$.
Note that non-trivial elements of $G_\infty$ (if any) are parabolic elements of $\PC$ fixing $[\zeta_\infty:\xi_\infty] \in \CP1$.
\end{thm}
\begin{proof}
This proposition will be proved by applying the Reduction Lemma twice.

First, let us define for all $n$ a map $\psi_n: \CC \to \PC$ by 
  $$ z \mapsto (1+z/R_n)^{-1/2} \begin{pmatrix} 1 + z e^{i \omega_n} (\zeta_2)_n (\xi_1)_n &  z e^{i \omega_n} (\zeta_1)_n (\zeta_2)_n \\
   - z e^{i \omega_n} (\xi_1)_n (\xi_2)_n & 1- z e^{i \omega_n} (\zeta_1)_n (\xi_2)_n \end{pmatrix}$$

 Let us also define $\psi : \CC \to \PC$ by 
 $$z\mapsto \begin{pmatrix} 1 + z e^{i \omega_\infty} \zeta_\infty \xi_\infty &  z e^{i \omega_\infty} \zeta_\infty^2 \\
   - z e^{i \omega_\infty} \xi_\infty ^2 & 1- z e^{i \omega_\infty} \zeta_\infty \xi_\infty \end{pmatrix} $$
We need to verify that this family satisfies the condition of the Reduction Lemma. This will be done through the following lemmas. 
\begin{lem}
\label{lem:1.1}
 $\psi$ is proper and continuous. 
\end{lem}
\begin{proof}
This is clear, since whenever $z\to \infty$, at least one of the four entries in the matrix $\psi(z)$ blows off to infinity.
\end{proof}

\begin{lem}
\label{lem:1.2}
 $\psi_n$ converges uniformly to $\psi$ on every compact set. 
\end{lem} 
\begin{proof}
  It is sufficient to prove it for every compact $K_M=\{z\in \CC;\; |z|\leq M\}$.  
 Fix some $\epsilon>0$.  Since $R_n\to \infty$, we can find for every $M>0$ some integer $N$ such that, for all $n\geq N$ and all $z\in K_M$,
\[
 1-\epsilon\leq  (1+z/R_n)^{-1/2}  \leq 1+\epsilon
\]
Therefore, we can also find an integer such that for every $n$ larger than this $N$,
\[
 \|\psi_n(z)-\psi(z)\| \leq \epsilon
\]
holds for every $z\in K_M$, thus we are done. 
\end{proof} 

\begin{lem} 
\label{lem:1.3}
 For any compact subset $K$ of $\PC$, the closed subset of $\CC$ $$\overline{ \bigcup_{n \ge N} \psi^{-n}(K) }$$ is compact for $N$ large enough. 
\end{lem} 
\begin{proof} 
It is sufficient to prove that for every $M>0$ and for every $z$ with $|z|>M$, one of the entries of $\psi_n(z)$ has a modulus greater that some quantity $A(M)$ depending only on $M$, 
with $A(M) \to \infty$ as $M \to \infty$. 
 Recall that $  (\zeta_2)_\infty (\xi_1)_\infty $, $(\zeta_1)_\infty (\zeta_2)_\infty$, $ (\xi_1)_\infty (\xi_2)_\infty$ and $ (\zeta_1)_\infty (\xi_2)_\infty$ cannot vanish at the same time. 
 Suppose for instance that the first entry does not vanish (other cases are similar). 
 Then there is a constant $c>0$ for which we have $|(\zeta_2)_n (\xi_1)_n| \ge c$ for every $n$ larger than some integer $N$.
 Also by taking a larger $N$ if necessary, we may assume that $R_n \ge M$.
  Thus we have 
 $$ |(1+z/R_n)|^{-1/2} |1 + z e^{i \omega_n} (p_2)_n (q_1)_n | \ge \frac{|z| c - 1}{ \sqrt{|z|} \sqrt{ 1/|z| + 1/R_n  } }$$
 $$ \ge  \frac{|z| c - 1}{ \sqrt{2|z| / M} } \ge \sqrt{M} \frac{|z| c - 1}{\sqrt{2 |z|}} \ge (c/\sqrt{2}) M -\frac{1}{\sqrt{2}}$$  
Thus we are done. 
\end{proof} 

Define for all $n$ the subset $\FF_n\subset \CC$ by
\begin{align*}
\FF_n&=\left\{ \dfrac{a-1}{|(\zeta_2)_n(\xi_1)_n-(\zeta_1)_n(\xi_2)_n|} ;\; a \in \Xi_n\right\}\\
           &=\left\{ R_n(a-1) ;\; a \in \Xi_n\right\}
\end{align*}
Putting together Lemmas \ref{lem:1.1}, \ref{lem:1.2}, \ref{lem:1.3} and Proposition \ref{redlem}, we now see that if $\FF_n$ converges to some closed subset $\FF_\infty\subset \CC$ for the Hausdorff topology, then $G_n=\psi_n(\FF_n)$ converges to $\psi(\FF_\infty)$ for the Chabauty topology, hence $G_\infty=\psi(\FF_\infty)$.
We are done for the first step of the proof.

The second step consists of applying the Reduction Lemma again.
For all $n$, define $\varphi_n: \CC \rightarrow \CC$ by
\[
 \varphi_n(z)=R_n(e^{z/{R_n}}-1)
\]
By the series expansion of $\exp$, we see that the sequence $(\varphi_n)$ converges to the identity map of $\CC$, uniformly on every compact subset, which is obviously continuous and proper.
Since $\varphi_n$ is periodic of period $2i\pi R_n$, there is no chance for
\[
\overline{\bigcup_{n\geq N}\varphi_n^{-1}(K)}
\]
to be ever compact. 
Thus, let us use Remark \ref{rmk:omega}, by defining for all $n$ $\Omega_n$ to be the band
\[
 \Omega_n=\{z\in \CC;\; |\IM(z)|\le \pi R_n \}
\]
which verifies the required conditions of Remark \ref{rmk:omega}, because $R_n\to \infty$.

Then, we have the following:
\begin{lem}
For every $M>0$, for every $z\in \Omega_n$, $\|z\|_\infty \ge M$ implies $|\varphi_n(z)|\ge M/2$ provided $n$ is large enough.
\end{lem}
\begin{proof}
Let us suppose $n$ is large enough so that $R_n\ge M$.
 If $|\RE(z)|\ge M$, then $|\varphi_n(z)|\ge R_n(|e^{z/{R_n}}|-1)\ge R_n(e^{M/{R_n}}-1) \ge M$.
Thus, let us suppose $z$ is in the closed subset
\[
 U_n=\{z\in \CC;\; |\RE(z)|\le M,\, M\le |\IM(z)|\le \pi R_n\} 
\]
Now $U_n$ is mapped homeomorphically by $[z\mapsto e^{z/{R_n}}]$ onto a horseshoe 
\[
 \{z\in \CC;\; e^{-M/{R_n}}\le|z|\le e^{M/{R_n}},\, M/{R_n}\le\Arg{z}\le 2\pi-M/{R_n}\}
\]
that avoids an open ball of radius $\sin(M/{R_n})$ around 1.
Thus $\varphi_n(U_n)$ avoids a ball of radius $R_n \sin(M/{R_n})$ around 0.
In view of the series expansion of $\sin$, and since $R_n\to \infty$, we must have $R_n \sin(M/{R_n})\ge M/2$ for $n$ large enough, and we are done.
\end{proof}

Thus, applying Proposition \ref{redlem} again, we see that if $\Gamma_n\to \Gamma_\infty$, then $\FF_n=\varphi(\Gamma_n)\to \Gamma_\infty$, hence $\FF_\infty=\Gamma_\infty$. 

All in all, we have $G_\infty=\psi(\Gamma_\infty)$, and this completes the proof of Theorem \ref{thm:reducingargument1}.
\end{proof}

\begin{thm}
\label{thm:reducingargument2}
Let $G_n$ as above.

Suppose that $R_\infty<\infty$.
Then the geometric limit of $G_n$ is the subgroup
 \[
  G_\infty=\left\{  H_{((\zeta_1)_\infty,(\xi_1)_\infty),((\zeta_2)_\infty,(\xi_2)_\infty),a};
  \; a \in \Xi_\infty\right\}
 \]
\end{thm}
\begin{proof}
We leave to the reader to check that we can apply the Reduction Lemma as before, with $\psi_n :\CC^\ast \to \PC$ defined by
$$ z \mapsto z^{-1/2} \begin{pmatrix} 1 + (z-1)R_n e^{i \omega_n} (\zeta_2)_n (\xi_1)_n &  (z-1)R_n e^{i \omega_n} (\zeta_1)_n (\zeta_2)_n \\
   - (z-1)R_n e^{i \omega_n} (\xi_1)_n (\xi_2)_n & 1- (z-1)R_n e^{i \omega_n} (\zeta_1)_n (\xi_2)_n \end{pmatrix} $$
   and  $\psi_\infty : \CC^\ast \to \PC$ by
$$ z \mapsto z^{-1/2} \begin{pmatrix} 1 + (z-1)R_n e^{i \omega_\infty} (\zeta_2)_\infty (\xi_1)_\infty &  (z-1)R_n e^{i \omega_\infty} (\zeta_1)_\infty (\zeta_2)_\infty \\
   - (z-1)R_n e^{i \omega_\infty} (\xi_1)_\infty (\xi_2)_\infty & 1- (z-1)R_n e^{i \omega_\infty} (\zeta_1)_\infty (\xi_2)_\infty \end{pmatrix} $$

We can therefore conclude that since $\Xi_n$ converges to $\Xi_\infty$, $G_n=\psi_n(\Xi_n)$ converges to $\psi(\Xi_\infty)$;
thus $G_\infty=\psi(\Xi_\infty)$, and we are done.
\end{proof}

\subsection{Geometric view of $R_n$ and $\omega_n$}
\label{subsec:geometricviewRomega}
Let us give geometric interpretations for $R$ and $\omega$, defined as above by 
\[
 R e^{i\omega}=\dfrac{1}{\zeta_2\xi_1-\zeta_1\xi_2}
\]
where $(\zeta_1,\xi_1),(\zeta_2,\xi_2)\in \Dp$.
We already saw in Subsection \ref{subsec:Theta} that $\zeta_2\xi_1-\zeta_1\xi_2$ can be interpreted as a square rooted discriminant in a model of $\CPP$. 

\begin{lem}
$R$ is the inverse of the spherical distance between $[\zeta_1:\xi_1]$ and $[\zeta_2:\xi_2]\in \CP1$.
More precisely, $1/R=|\zeta_2\xi_1-\zeta_1\xi_2|$ is equal to the distance, in $\RR^3=\CC\times\RR$, between the respective stereographic projections of $(\zeta_1,\xi_1)$ and $(\zeta_2,\xi_2)\in \Dp$ on the sphere $S^2$ of center $(0,1/2)$ and radius $1/2$.
\end{lem}
\begin{proof}
Recall from Remark \ref{rmk:stereo} that the stereographic projection of $\Dp$ onto $S^2\setminus S$ is given by
\[
(\zeta,\xi)\mapsto (\zeta\xi,\xi^2)
\]
Now, using that $|\zeta_i|^2=1-\xi_i^2$, it is a straightforward computation to show that 
 \[
  |\zeta_2\xi_1-\zeta_1\xi_2|^2=|\zeta_2\xi_2-\zeta_1\xi_1|^2+(\xi_1^2-\xi_2^2)^2.
 \]
\end{proof}

\begin{lem}
$\omega$ is the opposite of the angle between $[\zeta_1:\xi_1]$ and $[\zeta_2:\xi_2]\in \CP1$.
More precisely, $-\omega=\Arg(\zeta_2\xi_1-\zeta_1\xi_2)$ is equal to the argument, in $\CC\times \{1\}\equiv\CC$, between the respective stereographic projections of $(\zeta_1,\xi_1)$ and $(\zeta_2,\xi_2)\in \Dp$ on the horizontal plane passing through $N=(0,1)$.
\end{lem}
\begin{proof}
Since multiplying a number by a positive real does not change the argument,  
\[
-\omega=\Arg(\zeta_2\xi_1-\zeta_1\xi_2)=\Arg(\zeta_2/\xi_2-\zeta_1/\xi_1).
\]
\end{proof}

\begin{figure}[ht]
\begin{center}
\includegraphics[scale=0.4]{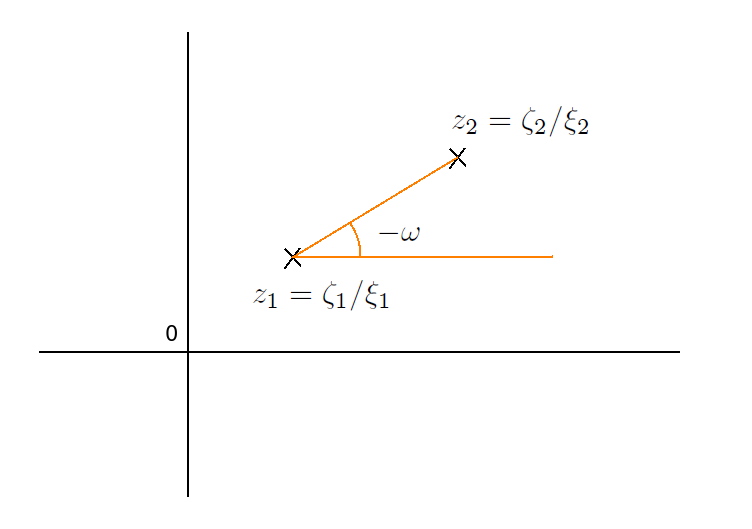}
\end{center}
\caption{$-\omega$ as an angle in $\CC\times\{1\}$.}
\label{fig:omega}
\end{figure}

\subsection{Geometric limits seen with cylinders}
\label{subsec:cylinders}
We would like now to give a geometric interpretation of Theorem \ref{thm:reducingargument1}.
Let $(G_n)$ be a sequence of non-trivial non-parabolic closed abelian subgroups of $\PC$ as above. 
Recall notations for $((\zeta_1)_n,(\xi_1)_n),((\zeta_2)_n,(\xi_2)_n)\in \overline{\Dp}$, $R_n\ge1$, $\omega_n \in [0,2\pi]$.
Assuming here that $R_n\to \infty$, i.e. that the distance between the fixed points $[(\zeta_1)_n:(\xi_1)_n],[(\zeta_2)_n:(\xi_2)_n]\in \CP1$ of $G_n$ tends to 0,
recall the notation $\Gamma_n=R_n \Log(\Xi_n)$ with $\Xi_n$ the multiplicative group of rotation/complex translation quantities of elements of $G_n$ (see beginning of Section \ref{sec:redlem}).

For all $n$, $\Gamma_n$ is a subgroup of $\CC$ containing $2i\pi R_n$. 
Equivallently, $\Gamma_n/2i\pi R_n\ZZ$ is a subgroup of $\CC/2i\pi R_n\ZZ$, which is a cylinder.
Let us view this cylinder in $\RR^3=\CC\times\RR$ as being the cylinder with circumference $2i\pi R_n$ (i.e. radius $R_n$) and with center line of equation 
\[
 \{(z,t)\in\CC\times \RR;\; \IM(z)=0 \text{ and } t=R_n \}
\]
This cyclinder intersects the plane $\CC\times \{0\}$ in the line $\{\IM(z)=0\text{ and } t=0\}$.

Otherwise put, this cyclinder is the image of $\CC$ under the map 
\begin{align*}
 \CC\to &\CC\times \RR\\
x+iy\mapsto &(x+iR_n \sin y, R_n(1-\cos y))
\end{align*}

Better yet, imagine this cylinder as being rotated by an angle $\omega_n$, as in the following drawing, Figure \ref{fig:cylinders}. 
For notational convenience, let us note this cylinder $C_n$.

\begin{figure}[ht]
\begin{center}
\includegraphics[scale=0.25]{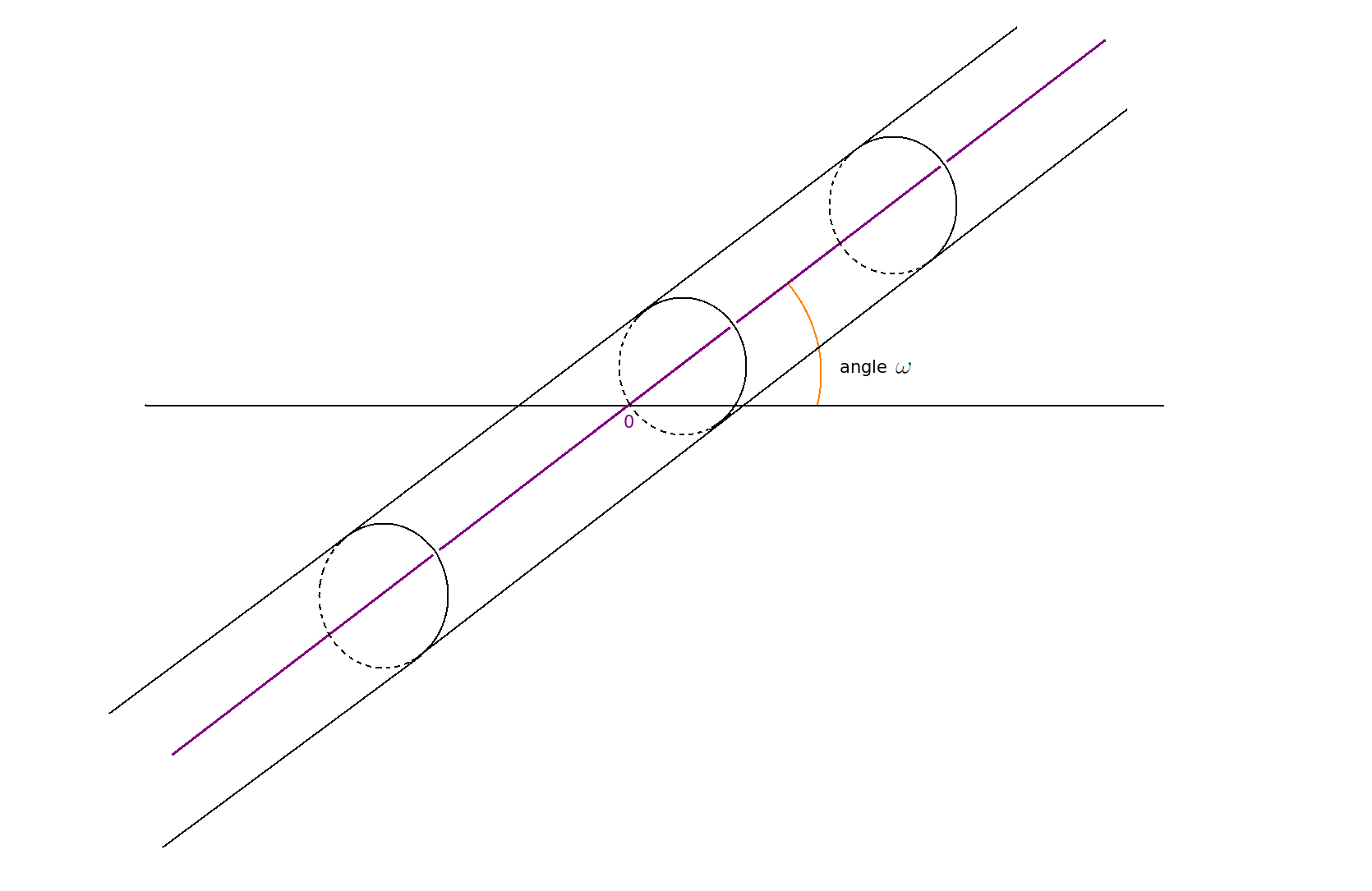}
\end{center}
\caption[Cylinder $C_n$ associated to $G_n$]{Cylinder $C_n$ associated to $G_n$, with in purple the line corresponding to the intersection of $C_n$ with the plane $\CC\times \{0\}$.}
\label{fig:cylinders}
\end{figure}

Now when $n\to \infty$, $R_n\to \infty$ also, i.e. the cylinders $C_n$ become wider and wider; therefore $(C_n)$ converges, for the Hausdorff topology of $\RR^3$, to the plane $\CC\times \{0\}$. 

The last step of the description is to draw for all $n$ the subgroup $\Gamma_n/2i\pi R_n\ZZ$ on $C_n$, simply as the image of $\Gamma_n$ under the map
\begin{align*}
 \CC\to &\CC\times \RR\\
x+iy\mapsto &(e^{i\omega_n}(x+iR_n \sin y), R_n(1-\cos y))
\end{align*}

As $n\to \infty$ and the cylinders $C_n$ become wider and wider, these images look more and more like a closed subgroup of $\CC\times \{0\}$, that we recognize to be $e^{i\omega_\infty}\Gamma_\infty$.

Finally, plug the values of this subgroup $e^{i\omega_\infty}\Gamma_\infty\subset \CC$ in the matrix representation so that we obtain 
\[
\left\{ \begin{pmatrix} 1 + \rho\, \zeta_\infty \xi_\infty & \rho\, \zeta_\infty^2 \\
   -\rho \xi_\infty^2 & 1- \rho\, \zeta_\infty \xi_\infty \end{pmatrix};\; \rho \in e^{i\omega_\infty}\Gamma_\infty
\right\}
\]
which is the geometric limit $G_\infty $ of $G_n$.

\pagebreak

\section{Case $R_\infty<\infty$: the Chabauty space of $\CC^\ast$}
\label{sec:cstar}
In this section, we study the space $\CCC(\CC^*)$ of all closed subgroups of the multiplicative group $\CC^*$ in the Chabauty topology. 
This is justified by the reducing argument Theorem \ref{thm:reducingargument2}.

Note that $\CC^*$ is a conformal annulus with infinite modulus, hence there is a conformal isomorphism (namely the logarithm function) between $\CC^*$ and $\CC / 2 i\pi \ZZ$.
It induces a homeomorphism between $\CCC(\CC^*)$ and $\CCC(\CC / 2 i\pi \ZZ)$.
Considering the natural covering map $\pi: \CC \to \CC / 2 i\pi \ZZ$, we can identify $\CCC(\CC / 2 i\pi \ZZ)$ and the space $\WW(2\pi)\subset \CCC(\CC)$ of closed subgroups of $\CC$ that contain $2i\pi$, 
so that studying $\CCC(\CC^*)$ is equivalent to studying $\WW(2\pi)$ in the geometric topology.

The reason why we prefer to work with $\WW(2\pi)$ instead of $\CCC(\CC^*)$ is twofold. 
First, powers of an element $z\in \CC^*$ all lie in some logarithmic spiral; as a contrast, multiples of an element $z\in \CC$ lie in a line; thus it is in general easier to visualize geometric behaviors in the latter space than in the former.
Second, and more importantly, some of the result in this section can be directly transposed in the context of Theorem \ref{thm:reducingargument1}. 
Thus, let us introduce the following definition, where we think of $l$ to be $2\pi R_n$, with $R_n$ the inverse of the spherical distance between to points of $\CP1$ (see Subsection \ref{subsec:cylinders}).

\begin{defn}
\label{defn:wwl}
For every $l>0$, define $\WW(l)$ to be the space of all closed subgroups of $\CC$ that contain $l i$.
$\WW(l)$ is equipped with the Chabauty topology. 
\end{defn}

Using this notation, and in view of Theorems \ref{thm:reducingargument1} and \ref{thm:reducingargument2}, let us spell out what questions we wish to answer.

\begin{ques}
\label{ques:q1}
For a fixed $l$ (for instance $l=2\pi$), what is the space of geometric limits of sequences of elements of $\WW(l)$?
Equivallently, what is the Chabauty space $\CCC(\CC^*)$ of $\CC^*$?
\end{ques}

\begin{ques}
\label{ques:q2}
What is the space of geometric limits of sequences $(\Gamma_n)$, with $\Gamma_n\in\WW(l_n)\subset \CCC(\CC)$ for all $n$, and with $l_n\to \infty$?
\end{ques}

Question \ref{ques:q1} will be answered shortly. 
We would like to point out that the description in Subsection \ref{subsec:wwl} is already known, and can be found in \cite{Ismael}.
We include a proof performed in our ``linearized'' context, both for making the present paper self-contained, and because we believe it is  enlightenning for answering Question \ref{ques:q2}.
Pictures in Subsection \ref{subsec:pict} are new.

\subsection{$\WW(l)$}
\label{subsec:wwl}

Recall that any closed subgroup of $\CC$ is isomorphic to exactly one of the following groups: $\{0\}$, $\ZZ$, $\ZZ^2$, $\RR$, $\ZZ \times \RR$, $\CC$.





\begin{lem}
\label{lem:subCl}
Let $l>0$. The followings are all the closed subgroups of $\CC$ containing $l$. 
\begin{itemize} 
\item $A^{l/m} := (l/m) i \ZZ$ for some $m \in \NN$,
\item $B_z^{l/m} :=  z\ZZ + (l/m)i\ZZ$ for some $m \in \NN$ and for $z \in \CC$ with $\RE(z) >0$ and $\IM(z) \in [0,l/m]$, 
\item $C_x := x \ZZ + i \RR$ for $x >0$, 
\item $D_{t}^{l/m} := (l/m)i\ZZ + (1+i t)\RR$ with $t \in \RR$ and $m \in \NN$.
\item $A^0 = C_\infty := i \RR$, 
\item $C_0 := \CC$,  
\end{itemize}  
\end{lem} 
\begin{proof}
Let $\Gamma$ be a closed subgroup of $\CC$ containing $il$, i.e. $\Gamma\in \WW(l)$. Then $\Gamma$ contains $A^l$.
Therefore if $\Gamma$ is discrete, it must contain $ A^{l/m}$ for some maximal $m$;
if $\Gamma$ is isomorphic to $\RR$, it must be $C_\infty=i\RR$.

Also, if $\Gamma$ is a lattice containing $ A^{l/m}$ for a maximal $m$, it must be of the form $B_z^{l/m}$ for $z$ verifying $\RE(z) >0$ and $\IM(z) \in [0,l/m]$.

Thus, suppose $\Gamma$ is isomorphic to $\RR \times \ZZ$. There are two cases. 

Case 1: $A^{l/m}$ is the $\ZZ$ part. Then $\Gamma$, as a set, is the union of parallel lines of finite slope $t$ passing through the points of $A^{l/m}$. Thus $\Gamma$ is $D_t^{l/m}$. 

Case 2: $A^{l/m}$ is contained in the $\RR$ part. Then $\Gamma$ contains $C_{\infty}$ and is the union of vertical lines equally spaced in the horizontal direction. 
Hence $\Gamma = C_x$ for some $x>0$. 

Otherwise, $\Gamma=\CC$. 
\end{proof}

\begin{rmk}
\label{rmk:translationtable}
Let $G$ be a non-trivial non-parabolic closed abelian subgroup of $\PC$; let $\Xi_n$ be the multiplicative group of rotation/complex translation quantities of elements of $G$, and let $\Gamma=R\Log\Xi$ (see beginning of Section \ref{sec:redlem} for notations).
Exactly one of the following holds, with $l=2\pi R$ throughout:
\begin{itemize} 
 \item $\Gamma=A^{l/m}$ if $G$ is generated by an elliptic element of order $m$,
 \item $\Gamma$ is some $B_z^{l/m}$ if $G$ is generated by an elliptic element of order $m$ and a non-trivial hyperbolic element;
these two generators need to have the same fixed points in $\CP1$ in order for $G$ to be abelian. 
 \item $\Gamma$ is some $C_x$ if $G$ contains every elliptic element fixing the same two points in $\CP1$, 
and $G$ contains a non-trivial hyperbolic element fixing these same two points.
 \item $\Gamma$ is some $D_{t}^{l/m}$ if $G$ contains exactly $m$ elliptic elements, and has exactly $m$ connected components homeomorphic to $\RR$. Otherwise put, $\Xi$ is a $m$-branched logarithmic spiral.
 \item $\Gamma= i \RR$ if $G$ consists of every elliptic elements fixing the same fixed points, 
 \item $\Gamma= \CC$ if $G$ contains every elliptic and hyperbolic elements fixing the same fixed points. 
\end{itemize}  

\end{rmk}





Lemmas \ref{lem:chabeasy} and \ref{r2limits} below answer Question \ref{ques:q1}, with $l=2\pi$.

\begin{lem}
\label{lem:chabeasy}
 In the Chabauty topology, we have the following convergence results. 
 \begin{itemize}
 \item   $A^{2\pi/m_n} \to
 \begin{cases}
  A^0 = i \RR \text{ if } m_n \to \infty \\
  A^{2\pi/m} \text{ if } m_n \to m \in \NN 
 \end{cases}$
 \item   $ B_{z_n}^{2 \pi /m_n} \to
 \begin{cases}
  C_x \text{ if } m_n \to \infty \text{ and } \RE(z_n)\to x\\
  A^{2\pi/m} \text{ if } m_n \to m \in \NN \text{ and } \RE(z_n)\to \infty\\
  B_z^{2\pi/m} \text{ if } m_n \to m \in \NN \text{ and } z_n\to z \text{ with }\RE(z)\in (0,\infty) \\
 \end{cases}$
 \item  $C_{x_n} \to
 \begin{cases}
  C_0 =\CC \text{ if } x_n \to 0 \\
  C_\infty = i \RR \text{ if } x_n \to \infty \\
  C_x \text{ if } x_n \to x \in (0, \infty)
  \end{cases}$
 \item  $D_{t_n}^{2\pi / m_n} \to
 \begin{cases}
  \CC \text{ if } m_n \to \infty \text{ or } t_n \to \pm \infty \\
  D_t^{2\pi / m} \text{ if } t_n \to t \in \RR  \text{ and } m_n \to m \in \NN
  \end{cases}$
 \end{itemize}  
\end{lem} 
\begin{proof} 
The proofs of these assertions are either easier than or similar to the proof of Lemma \ref{r2limits} below; they are therefore left to the reader to check.
\end{proof} 

\begin{lem}
\label{r2limits}  
Let $(z_n=x_n+ 2i\pi\theta_n)$ be a converging sequence of complex numbers, with $x_n>0$, $x_n\to 0$, $\theta_n\in [0,1]$, $\theta_n\to \theta$. 
If $\theta$ is rationnal, say $\theta=p/q$ with $p$, $q$ coprime positive integers, 
define for all $n$ $t_n$ to be the slope of the line passing through 0 and $q z_n -2i\pi p$, i.e.
\[
t_n=\dfrac{2\pi}{x_n}(\theta_n-\theta ).
\]
Then the limit in the Chabauty topology of the sequence $ B_{z_n}^{2 \pi /m} $ is 
\[
 \begin{cases}
  D_t^{2 \pi / \lcm(m,q)} \text{ if } t_n\to t \in \RR \\
  \CC \text{ if } t_n\to \pm \infty 
 \end{cases}
\]

If $\theta$ is irrationnal then $ B_{z_n}^{2 \pi /m} \to \CC$.
\end{lem}
\begin{proof}
The case where $\theta$ is irrationnal is immediate; let us suppose that $\theta\in \QQ$.

Note, just by drawing all lines of slope $t_n$ passing through points of $B_{z_n}^{2 \pi /m}$, that $B_{z_n}^{2 \pi /m} \subset D_{t_n}^{2 \pi / \lcm(m,q)} $.
Moreover, a closer look on the intersection between all those lines and the imaginary axis $i \RR$ shows that on every line of $D_{t_n}^{2\pi / \frac{mq}{\gcd(pm,q)}}$ there is actually at least one point of $B_{z_n}^{2 \pi /m}$, as soon as $\theta_n$ is close enough to $\theta$.
Since $p$ and $q$ are coprime, $\frac{mq}{\gcd(pm,q)}=\lcm(m,q)$, thus there is at least one point of $B_{z_n}^{2\pi/m}$ on each line of $D_{t_n}^{2\pi / \lcm(m,q)}$.

Finally, since $x_n\to 0$, we can find for every $\epsilon>0$ an integer $N$ large enough so that for all $n\geq N$, $B_{z_n}^{2 \pi/m}$ $\epsilon$-fills $D_{t_n}^{2\pi / \lcm(m,q)}$
(i.e. every point of $D_{t_n}^{2\pi / \lcm(m,q)}$ is at distance at most $\epsilon$ of a point of $B_{z_n}^{2 \pi/m}$).
Therefore the Hausdorff distance between $B_{z_n}^{2 \pi/m}$ and $D_{t_n}^{2\pi / \lcm(m,q)}$ tends to zero, and we are done.
\end{proof}

\subsection{Pictures for the Chabauty space of $\CC^*$}
\label{subsec:pict}


Let us interpret the results of Subsection \ref{subsec:wwl} geometrically.

First, let us describe the space of subgroups $ D_t^{2\pi/m}$, for $m \in \NN$ and $t \in [-\infty, \infty]$.
By Lemma \ref{lem:chabeasy}, $D_t^{2\pi/m} \to \CC$ for any $m$ if $t \to \pm \infty$.
Thus we get a bouquet of circles; one circle for each $m \in \NN$, the wedge point corresponding to the total subgroup $\CC$.
We also know that when $m \to \infty$, $D_t^{2\pi/m} \to \CC$ for any $t$.
Thus when we increase $m$, the corresponding circle in the bouquet shrinks down to the wedge point.
We call this space the $D$-bouquet. See Figure \ref{fig:Dbouquet} below.

\begin{figure}[ht]
\begin{center}
\includegraphics[scale=0.6]{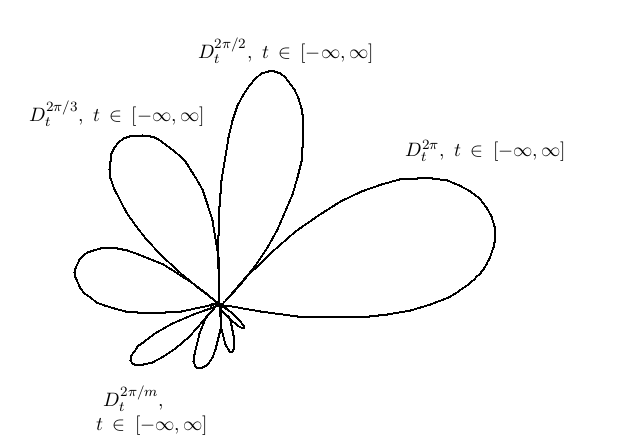}
\end{center}
\caption[$D$-bouquet in the Chabauty space of $\CC^*$]{The $D$-bouquet, a bouquet of circles; the wedge point represents the total subgroup $\CC$.}
\label{fig:Dbouquet}
\end{figure}

Now set $m \in \NN$ to be fixed.
We would like to see how the closure of the space of subgroups $B_z^{2\pi/m}$ looks like, for $z$ verifying $\RE(z)>0$ and $\IM(z)\in [0,2\pi/m]$.

But since two subgroups $B_z^{2\pi/m}$, $B_{z'}^{2\pi/m}$ for $z,z'$ as above, are the same if and only if $z=z' \mod 2i\pi/m $,
the space of $B_z^{2\pi/m}$'s is the cylinder 
\[
\{z=x+iy ;\; x >0 ,\, y \in [0, 2\pi/m]  \} \Big/ (x \sim x+ 2i\pi/m).
\]
 By Lemma \ref{lem:chabeasy}, if $x \to \infty$, then $B_z^{2\pi/m} \to A_m$ in the Chabauty topology.
 Therefore the space becomes a cone in the right direction.
 The identification of the other end is more complicated.
 Say $x \to 0$ and $y \to 2\pi \theta$ with $\theta\in [0,2\pi/m]$.
 By Lemma \ref{r2limits}, there are two cases to consider. 

Case 1: $\theta$ is irrational. Then $B_z^{2\pi/m}$ converges to $\CC$. 

Case 2: $\theta$ is rational, say $\theta=p/q$. Then $B_z^{2\pi/m} \to D_t^{2\pi/\lcm(m ,q)}$ where $t=\lim \frac{2\pi}{x_n}(\theta_n-\theta )$. 
For $p$ and $q$ fixed, every possible limit for $t$ is possible; hence we have to blow up the point $ 0+ 2i\pi p/q$ at the left of the cylinder to a segment corresponding to $D_t^{2\pi/\lcm(m ,q)}$ with $t \in [-\infty,\infty]$.
Now since $D_{\pm\infty}^{2\pi/\lcm(m ,q)}=\CC$, we still have to pinch the endpoints of that segment to a point, as in Figure \ref{fig:pinching} below.

\begin{figure}[ht]
\begin{center}
\includegraphics[scale=0.4]{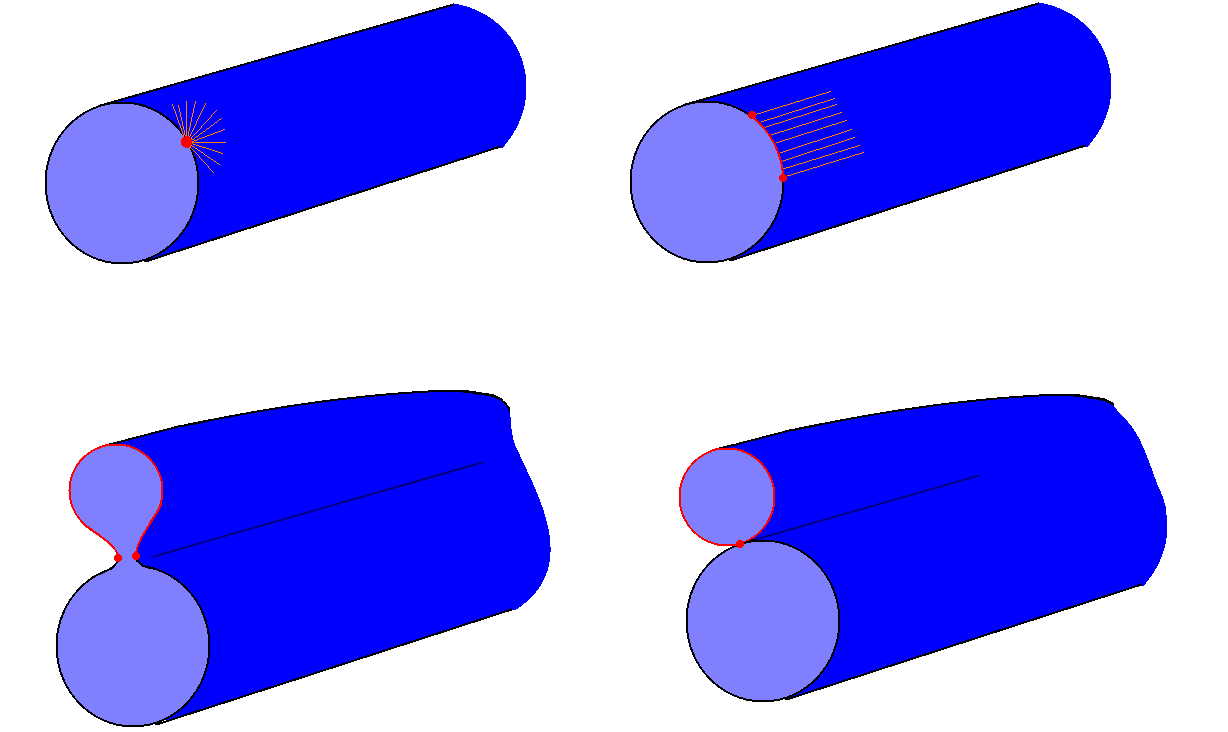}
\end{center}
\caption[Pinching in the $D$-bouquet]{At the left of the cylinder, each point $2i\pi\theta$ with $\theta$ rational is blown-up to a circle, resulting from a pinching.
First step: $2i\pi\theta$ (in red in the first picture) is blown-up to a segment (in red in the second picture). In orange, each ray represents a locus of constant slope $t_n$ (see Lemma \ref{r2limits} for notations). Second step: force the end-points of this segment to get closer and closer together (third picture) until the pinching (last picture). 
}
\label{fig:pinching}
\end{figure}

Therefore the left end of the cylinder, where $x = 0$, is glued to $D$-bouquet, in such a way that all points $2i\pi \theta$ with $\theta$ irrational are collapsed to the wedge point of the bouquet, and the other points are blown up to some circles of the bouquet.

Note that $\theta_1=p_1/q_1 $ and $\theta_2=p_2/q_2 $ are blown up to the \textit{same circle}, as long as $\lcm(m, q_1) = \lcm(m, q_2)$.
Thus if $m=1$, then this end is exactly the $D$-bouquet; but whenever $m > 1$, the end is glued to a proper subbouquet, containing only petals of index in $m\NN$. 

We call the resulting space the $m$th layer, noted $L_m$.

We can now collect every result of Lemmas \ref{lem:chabeasy} and \ref{r2limits} into a global picture for the Chabauty space of $\CC^*$, $\WW(2\pi)=\CCC(\CC^*)$ (see Figure \ref{fig:mlayer2}).

\begin{figure}[ht]
\begin{center}
\includegraphics[scale=0.3]{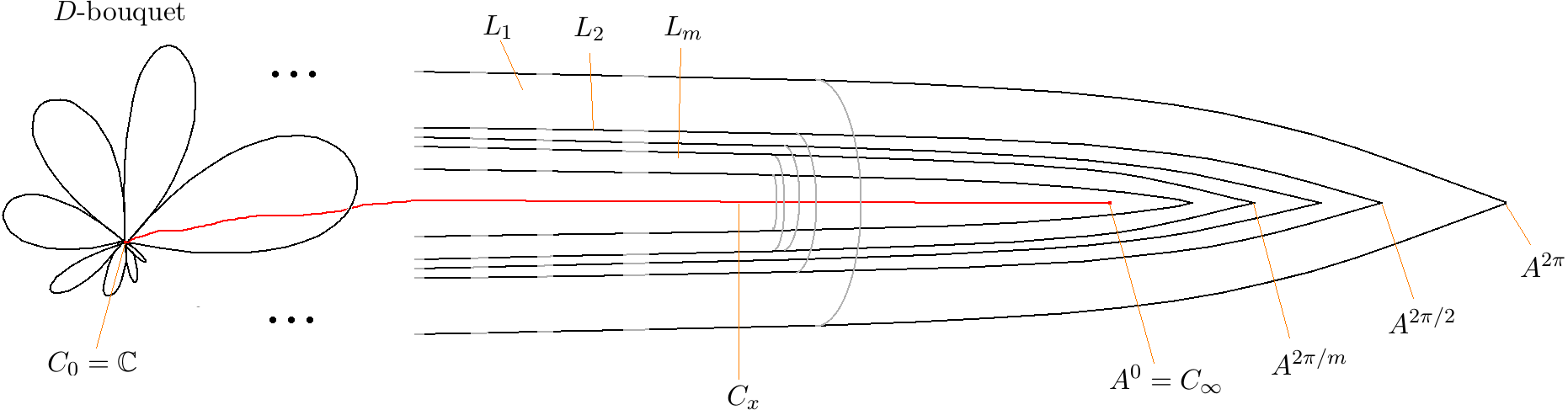}
\end{center}
\caption{The Chabauty space of $\CC^*$}
\label{fig:mlayer2}
\end{figure}

\pagebreak

\section{Case $R_\infty<\infty$: how the non-parabolic subgroups accumulate onto $\PS_2$}
\label{sec:bcstar}

In the previous section, we studied the Chabauty space of $\CC^*$, $\CCC(\CC^*)=\WW(2\pi)$ (see Definition \ref{defn:wwl});
along with Theorem \ref{thm:reducingargument2}, this describes the geometric limits of non-trivial non-parabolic closed abelian subgroups of $\PC$,
in the case where the limit is non-parabolic; using the notation developped in Section \ref{redlem}, this means $R_\infty <\infty$.

Let us now study the case where the limit is parabolic, i.e. $R_\infty=\infty$.
In view of Theorem \ref{thm:reducingargument1}, we only need to describe the limits of sequences of closed subgroups $\Gamma_n$ of $\CC$ such that $\Gamma_n\in \WW(2\pi R_n)$ for all $n$, with $R_n\to \infty$.
This was the content of Question \ref{ques:q2}.


Geometrically, $R_n\to \infty$ is classically viewed with an invariant cone getting wider and wider.
This view is due to Jorgensen, and is very well-explained in the lecture note of Jeff Brock (p.2 of \cite{Brock}).
Alternatively, we explained in Subsection \ref{subsec:cylinders} how to see it with cylinders $\CC / 2 i\pi R_n  \ZZ$ accumulating onto a plane.

\subsection{Some easy cases}
Let $\Gamma_n$ be a sequence of closed subgroups of $\CC$ as above, i.e. such that $\Gamma_n\in \WW(2\pi R_n)$ for all $n$.
If $i\RR\subset \Gamma_n$, define $l_n=1/n$.
Else, let $m_n$ to be the maximal integer such that $2i\pi R_n/{m_n}\in \Gamma_n$, and define $l_n=2i\pi R_n/m_n$.
Note that $\Gamma_n\in \WW(l_n)$ for all $n$;
we will always assume that the sequence $(l_n)$ converges in $[0,\infty]$, by extracting a subsequence if necessary.

Let us first deal with the case $l_n\to 0$.
For all $n$, define $p(\Gamma_n)\subset \RR$ to be the projection of $\Gamma_n$ to $\RR$ by $[z\mapsto \RE(z)]$.
It is a closed subgroup of $\RR$, with 
\begin{itemize}
 \item $p(\Gamma_n)=\{0\}$ if $\Gamma_n=A^{l_n}$ or $\Gamma=i\RR$,
 \item $p(\Gamma_n)=x\ZZ$ if $ \Gamma_n=B_z^{l_n}$ with $\RE(z)=x$ or $\Gamma=C_x$,
 \item $p(\Gamma_n)=\RR$ if $ \Gamma_n=D_t^{l_n}$ for some $t\in \RR$ or $\Gamma=\CC$.
\end{itemize}
See Lemma \ref{lem:subCl} for notations.

\begin{lem}
Let $\Gamma_n$ as above; suppose $l_n\to 0$.
Then if $p(\Gamma_n)$ converge in $\CCC(\RR)$ to a closed subgroup $K$, then $\Gamma_n$ converges in $\CCC(\CC)$ to $H+i\RR$.
\end{lem}
\begin{proof}
 This is easy and straightforward.
\end{proof}

Let us now see the case $l_n\to l$ with $l\in (0,\infty)$.
\begin{lem}
Let $\Gamma_n$ as above; suppose $l_n\to l$ with $l\in (0,\infty)$.
Define $\widetilde{\Gamma}_n=\frac{l}{l_n} \Gamma_n$.
Then $\Gamma_n$ and $\widetilde{\Gamma}_n$ have the same limit, $(\Gamma_n)$ converges.
\end{lem}
\begin{proof}
 This is an easy application of Proposition \ref{redlem}.
\end{proof}
Thus by Lemmas \ref{lem:chabeasy} and \ref{r2limits}, where the results are shown for $l=2\pi$, we are done in the case $l_n\to l$.

The following lemma deals with the easy convergence results in the case $l_n\to \infty$.
\begin{lem}
 In the Chabauty topology, we have the following convergence results, for $l_n\to \infty$ throughout.
 \begin{itemize}
 \item   $A^{l_n} \to \{0\}$
 \item   $ B_{z_n}^{l_n} \to
 \begin{cases}
  \{0\} \text{ if } \RE(z_n) \to \infty \\
  (x+iy)\ZZ \text{ if } \RR(z_n) \to x \text{ with } x>0 \text{, and } \widetilde{y_n}\to y \\
  \text{   where } \widetilde{y_n} \text{ is the representative of } \IM(z_n) \mod il_n \text{ in } (-l_n/2,l_n/2] \\
  \{0\} \text{ if } \RE(z_n) \to x \text{ with } x>0 \text{ and } \widetilde{y_n}\to \pm \infty\\
 \end{cases}$
 \item  $C_{x_n} \to C_x$ if $x_n\to x\in [0,\infty]$
 \item  $D_{t_n}^{l_n} \to 
 \begin{cases}
	  (1+it)\RR \text{ if } t_n\to t\in \RR\\
          C_x \text{ if } t_n\to \pm \infty \text{ and } \frac{|l_n|}{|t_n|}\to x\in [0,\infty]
 \end{cases}$
 \end{itemize}  \end{lem}
\begin{proof}
 The proofs of these assertions are elementary, and left to the reader to check.
\end{proof}

The last remaining case, namely the case $\Gamma_n= B_{z_n}^{l_n} $ with $l_n\to \infty$ and $\RE(z_n)\to 0$, will be studied separately in Subsection \ref{subsec:lastcase} below.
Note that, in view of Proposition 
\ref{redlem} and Remark \ref{rmk:translationtable}, understanding the convergence of $\Gamma_n$ in this special case is the same as understanding the convergence of sequences of  
closed abelian subgroups $G_n$ generated by an elliptic element of order $m_n$ (this condition being vacuous for $m_n=1$) and a non-trivial hyperbolic element, and such that $R_n/m_n\to \infty$.
In particular, Subsection \ref{subsec:lastcase} describes all possible geometric limits of sequences of cyclic groups of $\PC$ generated by one hyperbolic element, and converging to a parabolic group.

\subsection{Remaining case}
\label{subsec:lastcase}

We will now study the convergence of lattices $\Gamma_n \subset \CC$ of the form
\[
\Gamma_n= \langle il_n, x_n+i\theta l_n \rangle
\]
with $l_n>0$, $x_n>0$ and $\theta\in [0,1]$, in the case where $l_n\to\infty$ and $x_n\to 0$. 

Taking extractions if necessary, we can assume that $\theta_n\to \theta_\infty\in [0,1]$.



The strategy to describe the explicit limit of a sequence $(\Gamma_n)$ as above will be to replace the generators $( il_n, x_n+i\theta l_n)$ by a ``better'' pair of generators, ``better'' here meaning roughly ``closer to the origin''.
The intuition is that while $ il_n$ and $ x_n+i\theta l_n$ converge/diverge possibly in very different speeds,
linear combinations of these generators may very well end up close to the origin.
Better generators will prevent this enrichment behavior to happen.



This is where the continued fractions enter to the picture. 
Continued fractions appear in various parts of Mathematics in very interesting ways. 
The note \cite{AHTopNum} by Allen Hatcher, cited as reference, will not be used directly; 
nevertheless, we recommand it, and we believe that much insight can be gained through this geometry-flavoured exposition of continued fraction, using the Farey diagram.
Anything we need in the following is presented now.

\subsubsection{Continued fractions} 
\label{subsec:contfrac}
Recall that any number $\theta\in \RR$ can be written as
$$\theta= \alpha_0 + \frac{1}{ \alpha_1 + \frac{1}{\alpha_2 + \frac{1}{\alpha_3 + \cdots}  }   }$$
with $\alpha_0\in \ZZ$ and $\alpha_i\ge 1$ for $i\ge 1$.

Such expansions are finite for rational numbers, and infinite for irrational numbers.
We use the following more compact notation: 
$$\theta=[\alpha_0; \alpha_1, \alpha_2, \alpha_3, \ldots] $$

Recall that for $j$ less than the length of the expansion in continued fraction of $\theta$ (this condition being vacuous for $\theta$ irrational), the $j$th convergent
\[
 \frac{p_j}{q_j} = [\alpha_0; \alpha_1, \alpha_2, \ldots, \alpha_j]
\]
with $p_j$, $q_j$ coprime, verify the following properties.


\begin{lem} 
\label{lem:easypptycontfrac}
By convention, $p_{-1} = 1$, $p_0 = \alpha_0$, $q_{-1} = 0$, $q_0 = 1$. For all $j \ge 1$, 
\begin{itemize}
\item $p_j = \alpha_j p_{j-1} + p_{j-2}$, $q_j = \alpha_j q_{j-1} + q_{j-2}$,
\item $q_{j+1}p_j - q_j p_{j+1} = (-1)^j$,
\item $p_j/q_j$ alternates around $\theta$. More precisely, $\sign(\theta-p_n/{q_n})=(-1)^n$.
\item $|\theta-p_j/q_j|<1/q_j q_{j+1} <1/q_j^2$
\end{itemize}
\end{lem} 

The following proposition is folklore, but it is usually stated in a weaker version; thus we reproved it here.
\begin{prop}
Let $\theta\in \RR$.
 The denominators $q_j$ of the convergents of $\theta$ are those integers minimizing 
\[
 a\mapsto \inf_{p\in \ZZ}|a\theta-p|
\]
amongst smaller positive integers.
More precisely, if we define a non-increasing non-negative map $f$ by
\[
 f:a\mapsto \inf_{1\le i\le a}\inf_{p\in \ZZ}|i\theta-p|
\]
and by $f(0)=\infty$, then the set $\{a\ge 1;\;f(a)<f(a-1) \}$ is precisely $\{q_j;\;j\in \NN\}$.
\end{prop} 
\begin{proof}
 Define an extraction $(b_j)$ such that $\{a\ge 1;\;f(a)<f(a-1) \}=\{b_j;\;j\in \NN\}$.
The purpose is here to prove that $(b_j)=(q_j)$ as a sequence.
Let us begin by an easy lemma. 
\begin{lem}
\label{lem:easypptycontfrac2}
 If two coprime positive integers $p$ and $q$ verify $|\theta-p/q|<1/{q^2}$, then the two lowest values of
\[
a\mapsto \inf_{p\in \ZZ}|a\theta-p|
\]
for $a\in \{1,\ldots , q\}$ are obtained for $a=q$ and for $a=a_0 \in \{1,\ldots , q-1 \}$ defined by
\[
 a_0 p = -\sign(\theta-p/q) \mod q
\]
\end{lem}
\begin{proof}
This lemma should be clear, when thinking about the forward orbit of $\theta$ in $S^1=\RR/\ZZ$.

\end{proof}

Now combining Lemmas \ref{lem:easypptycontfrac} and \ref{lem:easypptycontfrac2}, we conclude that each $q_j$ is in $\{b_j;\;j\in \NN\}$; moreover, the definition of $a_0$ in Lemma 
\ref{lem:easypptycontfrac2} implies that for all $j$ the largest $b_k$ such that $b_k<q_j$ must be $q_{j-1}$. Since $q_0=1$, we are done.
\end{proof}

\subsubsection{Applications} 
\label{subsec:contfracappl}


Recall the previous notations
$\Gamma_n= \langle il_n, x_n+i\theta l_n \rangle$
with $l_n>0$, $x_n>0$, $\theta\in [0,1]$, $l_n\to\infty$, $x_n\to 0$ and $\theta_n\to \theta_\infty\in [0,1]$.
Let 
\[
\theta_n=[0;\alpha_{n,1},\ldots,\alpha_{n,j},\ldots]
\]
be the continued fraction expansion of $\theta_n$.

Define for all integers $n$ and $j$ 
\[
u_{n,j} =q_{n,j} x_n +i l_n(q_{n,j}\theta_n-p_{n,j})
\]

Note that we always have $\Gamma_n=\langle u_{n,j}, u_{n,j+1} \rangle$, since the imaginary part of $ u_{n,j}$ and $u_{n,j+1}$ are of opposite sign, and since the interior of the non-degenerate triangle constituted by these two points and 0 has no element of $\Gamma_n$.

As a first easy consequence of Subsection \ref{subsec:contfrac}, we have the following:

\begin{lem}
\label{lem:cor1contfrac}
If for some sequence $(n\mapsto j_n)$, we have $u_{n,j_n}\to u_\infty$ and $u_{n,j_n+1}\to v_\infty$, with
\begin{align*}
0<\RE(u_\infty)<\RE(v_\infty)<\infty,\\
0<|\IM(v_\infty)|<|\IM(u_\infty)|<\infty\\
\end{align*}
then $\Gamma_n$ converges for the Chabauty topology to $\Gamma_\infty=\langle u_\infty,v_\infty \rangle$.
\end{lem}
\begin{proof}
This is immediate, since $u_{n,j_n}$ and $u_{n,j_n+1}$ always generate $\Gamma_n$, and converge to $\RR$-linearly independent vectors $u_\infty$, $v_\infty$.
\end{proof}

At this point, it does not seem that the way we expressed $\Gamma_n$ using the continued fraction expansion of $\theta_n$ is by any mean more concrete that the use of the Weierstrass elliptic function in \cite{PourHubb}.

Contrary to this appearence, the two following lemmas show that a lot of the properties of the pair $(u_{n,j},u_{n,j+1})$ can be ``read'' in the continued fraction expansion of $\theta_n$.

\begin{lem}
As above, let 
\[
\theta_n=[0;\alpha_{n,1},\ldots,\alpha_{n,j},\ldots]
\]
be the continued fraction expansion of $\theta_n$.
Also, choose a sequence $(n\mapsto j_n)$.

Then, $\rho_n=\dfrac{\RE(u_{n,j_n+1})}{\RE(u_{n,j_n})}$ has continued fraction expansion 
\[
\rho_n=[\alpha_{n,j_n+1};\alpha_{n,j_n},\alpha_{n,j_n-1},\ldots, \alpha_{n,1}]
\]
and $\eta_n=\Bigl\lvert\dfrac{\IM(u_{n,j_n+1})}{\IM(u_{n,j_n})}\Bigr\rvert$ has continued fraction expansion 
\[
\eta_n=[0;\alpha_{n,j_n+2};\alpha_{n,j_n+3},\ldots].
\]

In other words, $\rho_n$ is obtained by reading the continued fraction of $\theta_n$ backwards, starting at the index $j_n+1$, and 
$\eta_n$ is obtained by reading the continued fraction of $\theta_n$ forwards, starting at the index $j_n+2$. 
\end{lem}
\begin{proof}
The first assertion is proven by a simple induction on $k$ for $\rho_{n,k+1}=q_{n,k+1}/q_{n,k}$,
since $\rho_{n,1}=q_{n,1}/q_{n,0}=\alpha_{n,1}$ and 
\[\rho_{n,k+1}=\dfrac{q_{n,k+1}}{q_{n,k}}=\dfrac{\alpha_{n,k+1}q_{n,k}+q_{n,k-1}}{q_{n,k}}=\alpha_{n,k+1}+\dfrac{1}{\rho_{n,k}}\]

Similarly for the second assertion, define for all $k$, $\eta_{n,k+1}=\Bigl\lvert\dfrac{q_{n,k+1}\theta_n-p_{n,k+1}}{q_{n,k}\theta_n-p_{n,k}}\Bigr\rvert$. 
Then $\eta_{n,0}=\theta_n$ and
\[\eta_{n,k+1}=-\dfrac{q_{n,k+1}\theta_n-p_{n,k+1}}{q_{n,k}\theta_n-p_{n,k}}=-\alpha_{n,k+1}+\dfrac{1}{\eta_{n,k}}\]
and we are done.
\end{proof}

\begin{lem}
\label{lem:bastardlem}
The set of lattices of $\CC$ generated by a pair of vectors $u,v\in \CC$ verifying  
\[
\begin{cases}
0<\RE(u)<\RE(v),\\
0<|\IM(v)|<|\IM(u)|,\\
\IM(u)\cdot \IM(v) <0,\\
\RE(v)/\RE(u)\in \RR\setminus\QQ
\end{cases}
\]
is dense in the space of closed subgroups of $\CC$, for the Chabauty topology.
\end{lem}
\begin{proof}
This follows from a standard argument, left to the reader.
\end{proof}

The following proposition translates in the world of $\PC$ as saying that we can obtain any parabolic group $P$ as a sequence of cyclic groups $H_n$ with hyperbolic generators.
Additionally, this remains true if we ask the fixed points of $H_n$ to converge radially.
More precisely, suppose for instance that the fixed point of $P$ is $0\in \Chat\cong\CP1$ and choose some preferred angle $\omega$.
Then we can find a sequence $H_n$ converging to $P$ with $\Fix(H_n)=\{0,f_n\}$ and $\Arg{f_n}=-\omega$ for all $n$.

\begin{prop}
\label{prop:cor2contfrac}
Let $\Gamma$ be any closed subgroup of $\CC$, and $\theta\in [0,1)$.
Then there exist sequences $l_n\to \infty$, $x_n\to 0$ and $\theta_n \to \theta\in [0,1]/(0\sim 1)$ such that the sequence of lattices
\[
\Gamma_n=\langle il_n, x_n +i\theta_n l_n\rangle
\]
converges to $\Gamma$ in the Chabauty topology
\end{prop}
\begin{proof}
First assume that $\theta\in [0,1]\setminus \QQ$ and that $\Gamma$ is generated by a pair of vectors $u,v\in \CC$ verifying  
\[
\begin{cases}
0<\RE(u)<\RE(v),\\
0<|\IM(v)|<|\IM(u)|,\\
\IM(u)\cdot \IM(v) <0,\\
\RE(v)/\RE(u)\in \RR\setminus\QQ
\end{cases}
\]
Define $\rho=\RE(v)/\RE(u)$ and $\eta=|\IM(v)/\IM(u)|$

Suppose for instance that $\IM(u)>0$, the case $\IM(u)<0$ being similar.
Also, define the following continued fraction expansions:
\[
\begin{cases}
\theta=[0;\alpha_1,\alpha_2,\ldots]\\
\rho=[\beta_0;\beta_1,\beta_2,\ldots]\\
\eta=[0;\gamma_1,\gamma_2,\ldots]
\end{cases}
\]
Here by assumption, the first two expansions are infinite, $\beta_0>0$, and the last expansion is either finite or infinite.

For all $n$, let us define $\theta_n$ by its continued fraction expansion:
\[
\theta_n=[0;\alpha_1,\alpha_2,\ldots,\alpha_n,\beta_n,\beta_{n-1},\ldots,\beta_0,\gamma_1,\gamma_2,\ldots]
\]
Also, this defines some $\rho_n$ and $\eta_n$ by 
\[
\begin{cases}
\rho_n=[\beta_0;\beta_1,\ldots,\beta_n,\alpha_n,\ldots,\alpha_1] \\
\eta_n=[0;\gamma_1,\gamma_2,\ldots]
\end{cases}
\]

Define the coprime positive integers $p_n$ and $q_n$ such that $p_n/q_n$ has the expansion 
\[
\theta_n=[0;\alpha_1,\alpha_2,\ldots,\alpha_n,\beta_n,\beta_{n-1},\ldots,\beta_0]
\]

Now, define $x_n$ and $l_n$ so that $x_n q_n=\RE(u)$ and $l_n (q_n\theta_n-p_n)=\IM(u)$.
Then by Lemma \ref{lem:cor1contfrac}, $\Gamma_n=\langle il_n, x_n +i\theta_n l_n\rangle$ converges to $\Gamma=\langle u,v \rangle$, and we are done for this case.

But now, by Lemma \ref{lem:bastardlem} and by a standard density and diagonal argument, we are done in all the remaining cases.
\end{proof}
\begin{rmk}
We actually gave explicit generators for $\Gamma_n$ only in the generic case where $\Gamma=\lim\Gamma_n$ is as in Lemma \ref{lem:cor1contfrac} and $\theta\notin \QQ$.
This can also be achieved for the other choices of $\Gamma$ and $\theta$ by some minor changes, left to the reader. 
\end{rmk}

Finally, let us describe explicitely the limit of a converging sequence $\Gamma_n $ as above, using only the sequences $(l_n)$, $(x_n)$ and the coefficients $\alpha_i$ of the expansion in continued fraction $\theta$.
Let us start with an easy lemma.

\begin{lem}
For any $n$, the two minimal values of 
\[
j\mapsto q_{n,j}x_n+il_n(q_{n,j}\theta_n-p_{n,j})
\]
for the max norm 
\[
 \|x+iy\|_\infty=\Max(|x|,|y|)
\]
are obtained for two consecutive integers.
\end{lem}
\begin{proof}
Since $(j\mapsto q_{n,j}x_n)$ is increasingly converging to $\infty$ and $(j\mapsto l_n|q_{n,j}\theta_n-p_{n,j}|)$ is decreasingly converging to 0,
\[
j\mapsto \| q_{n,j}x_n+il_n(q_{n,j}\theta_n-p_{n,j}) \|_\infty
\]
is first decreasing and then increasing.
\end{proof}

\begin{prop}
For all $n$, let $j_n$ be such that the two minimal values of
\[
j\mapsto q_{n,j}x_n+il_n(q_{n,j}\theta_n-p_{n,j})
\]
for the max norm are obtained for $j_n$ and $j_n+1$.
Define $u_n=u_{n,j_n}$ and $v_n=u_{n,j_n+1}$; suppose for instance $\IM(u_n)>0$ for all $n$ (the case $\IM(u_n)>0$ for all $n$ is similar, and we can assume either one of the two by taking an extraction if necessary).
Define $t_u=\lim \Arg{u_n}\in [0,\infty]$, $t_v=\lim \Arg{v_n}\in [-\infty,0]$, assuming these limits exist by taking an extraction if necessary.

If $t_u$ and $t_v$ are neither both 0 nor both $\pm \infty$, then the limit subgroup $\Gamma_\infty=\lim \langle u_n,v_n\rangle$ is the one we expect, namely
\[
 \Gamma_\infty=\Gamma_u+\Gamma_v
\]
with 
\[
\Gamma_u=\lim \langle u_n\rangle= \begin{cases}
                                   (1+it_u)\RR \text{ if } u_n\to 0\\
                                   u_\infty\ZZ \text{ if } u_n\to u_\infty\in \CC\\
                                   \{0\} \text{ if } u_n\to \infty\\
                                  \end{cases}
\]
and similarly for $v$.

If $t_u$ and $t_v$ are either both 0 or both $\pm \infty$, then:
\begin{itemize}
 \item $\Gamma_\infty= iy\ZZ+\RR$ if $t_u=t_v=0$, $u_n\to 0$ and $\frac{\IM(u_n)}{\RE(u_n)}\RE(v_n)+|\IM(v_n)|\to y$,
 \item $\Gamma_\infty=x \ZZ$ if $t_u=t_v=0$ and $u_n\to x\in \RR,\, x > 0$,
 \item $\Gamma_\infty= x\ZZ+i\RR$ if $t_u=+\infty$, $t_v=-\infty$, $v_n\to 0$ and $\IM(u_n)\frac{\RE(v_n)}{|\IM(v_n)|}+\RE(u_n)\to x$,
 \item $\Gamma_\infty=iy \ZZ$ if $t_u=+\infty$, $t_v=-\infty$ and $v_n\to -iy,\, y > 0$.
\end{itemize}
\end{prop}
\begin{proof}
The first part follows easily from the minimality of the generators $(u_n,v_n)$.
The two cases $t_u=t_v=0$ and $t_u=+\infty$, $t_v=-\infty$ are similar; let us prove only the result for the former case.

If $t_u=t_v=0$ and $u_n\to 0$, draw the line passing through 0 and $u_n$, and consider its intersection $iy_n$ with the vertical axis.
It is easy to see that $y_n=\frac{\IM(u_n)}{\RE(u_n)}\RE(v_n)+|\IM(v_n)|$, and since $u_n\to 0$ and $t_u=0$, we conclude that $\Gamma_\infty= iy\ZZ+\RR$.

If $t_u=t_v=0$ and $u_n\to x$ with $x>0$, then consider Figure \ref{fig:ultimatecases}.
By the minimality of the generators $(u_n,v_n)$ for the max norm, there can not be any element of $\Gamma_n = \langle u_n,v_n \rangle$ in the two left yellow squares.
As a consequence, there can not be any element of $\Gamma_n$ in any of the yellow-shaded region. 
Now $|\IM(v_n)|<\IM(u_n)$, so $\RE(v_n)$ must be bigger than the real part of the purple point, which is easily seen to be $\frac{\RE(u_n)}{\IM(u_n)}(\RE(u_n)-\IM(u_n))$.
Since $u_n\to x>0$ we conclude that $\Gamma_\infty= x \ZZ$, and we are done for all cases.
\end{proof}

\begin{figure}[ht]
\begin{center}
\includegraphics[scale=0.3]{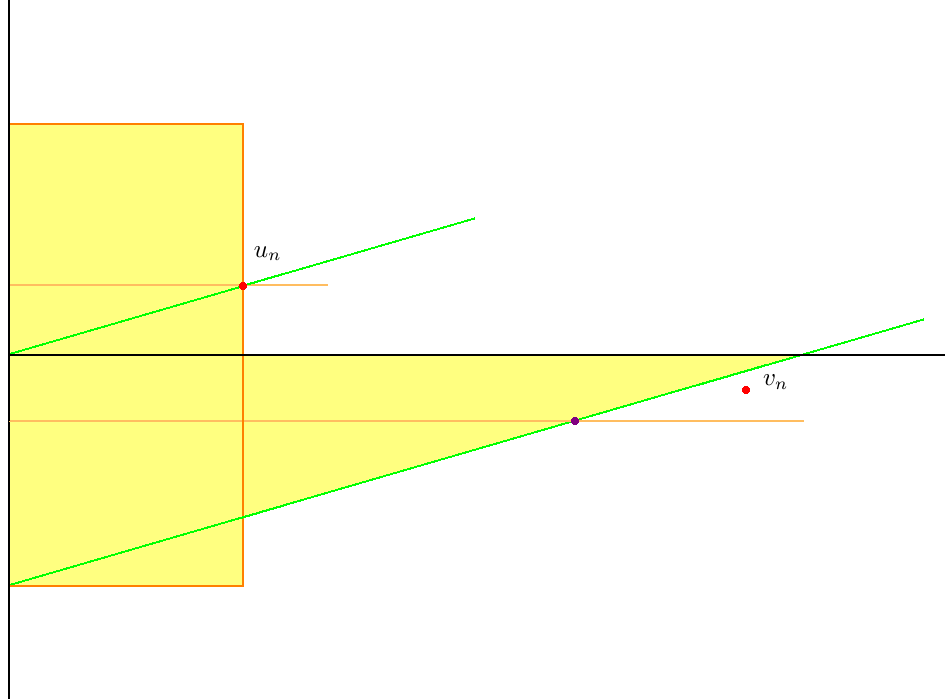}
\end{center}
\caption[One interesting deformation of $\ZZ^2$]{If $u_n\to x$ with $x\in \RR$, $x>0$, then $\RE(v_n)\to \infty$.}
\label{fig:ultimatecases}
\end{figure}

\pagebreak

\section{Local models for $\CS_2$}
\label{sec:pieces}
We would like now to provide local models for neighborhoods of elements in $\CCC_2$.
Recall that the space of non-trivial non-parabolic elements of $\CS_2$ is homeomorphic to $\Theta\times (\CCC(\CC^*)\setminus 1)$ (see Proposition \ref{prop:nonpar}).
Since we described geometrically $\CCC(\CC^*)$ in Subsection \ref{subsec:pict}, we have a clear enough picture of what a neighborhood of a non-trivial non-parabolic element of $\CS_2$ looks like.

Note that as this space accumulates to $\PS_2$, we face the situation of a 6-dimensional space accumulating on another 6-dimensional space. 
We expect spiraling behaviors of some sort; next subsection is an attempt to make this precise.

\subsection{Dichotomy of accumulation behavior} 

 Even though it might be hard to classify all the detailed cases of accumulation behaviors in general, there is one simple nice dichotomy for the case when an $n$-dimensional space $X$
accumulates to another $n$-dimensional space $Y$ (say $X$, $Y$ metric spaces). 
Let $p\in Y$ be a limit point of $X$. Then either there is a continuous path $\gamma : [0,1] \to X \cup Y$ such that $\gamma([0,1)) \subset X$ and $\gamma(1) = p \in Y$, or there is no such a path.
Otherwise put, either for every neighborhood $U$ of $p$ in $X \cup Y$ the arcwise-connected component of $U$ containing $p$ contains an element of $X$, or for every neighborhood $U$ of $p$ in $X \cup Y$ the arcwise-connected component of $U$ containing $p$ contains no element of $X$.
We would like to reserve the term ``spiraling of $X$ toward $Y$'' for the latter behavior, since it is similar to $[1,\infty)\subset\RR$ accumulating onto $S^1$ via $x\mapsto (1-\frac{1}{x})e^{ix}\in \CC$.
We do not think that this terminology is standard.

Let us see an example in dimension 2 showing the two different situations. See Figure \ref{fig:dichotomy}. 
\begin{figure}[ht]
\begin{center}
\includegraphics[scale=0.36]{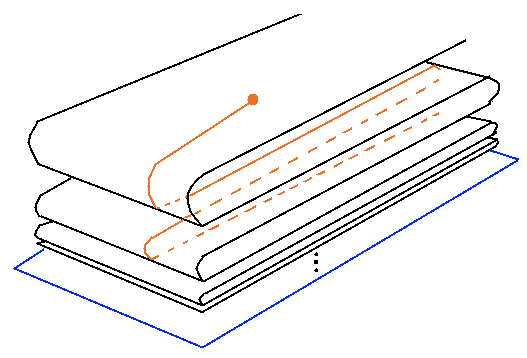}
\includegraphics[scale=0.36]{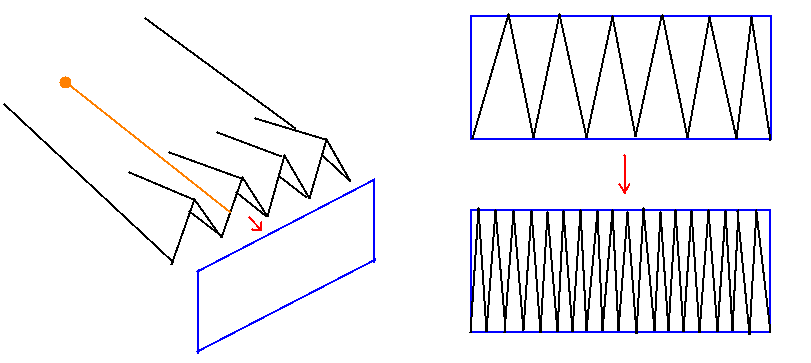}
\end{center}
\caption[A 2-dimensional example of the dichotomy ]{In both pictures, a 2-dimensional space $X$, in black, is accumulating onto a 2-dimensional space $Y$, 
in blue. In the picture below, we only show the approximation of a pleating of one end of $X$. It looks like an accordion with more and more pleating. As a limit of this process, 
$X$ finally accumulates onto a the square $Y$. The fundamental difference between these two cases is captured by the orange paths. In the first case, if you pick a point in $X$, represented as an orange dot in the picture, there is no finite path starting at this point and reaching $Y$. But such a path exists in the second case. }
\label{fig:dichotomy}
\end{figure} 

In fact, a mixture of the two behaviors shown in Figure \ref{fig:dichotomy} could arise. Still, the dichotomy holds: whenever these two behaviours appear at the same time, we can find a continuous path starting at a point of $X$ and reaching $Y$. 

For notational convenience, define $X$ to be the space of all non-trivial non-parabolic closed abelian subgroups of $\PC$. 
In the following subsections, we will prove in particular the following theorem. 

\begin{thm} 
\label{thm:DichotomyParabolics}
Let $G$ be a group in $\PS_2$. Then 
\begin{itemize}
 \item[1.] if $G$ is isomorphic to $\ZZ^2$, then $X$ accumulates towards $G$ in a spiraling way. 
 \item[2.] if $G$ is not, $X$ accumulates toward $G$ in a non-spiraling way. 
\end{itemize} 
\end{thm} 

With this is mind, let us now describe some models for neighborhoods $G$, depending of its isomorphy type.

\subsection{The spiraling case} 
 In this subsection, we prove the first case in Theorem \ref{thm:DichotomyParabolics}.
The intuition in this proof and the ones in Subsection \ref{subsec:nonspiraling} come from the following remark: non-parabolic groups in $\CS_2$ are specified by two fixed points and by a closed subgroup containing some element $2i\pi R_n e^{i\omega_n}$ with $R_n\ge 1$ (see Subsection \ref{subsec:cylinders}); alternatively, they can be specified by the giving of \textit{one} fixed point, and by a marked closed subgroup of $\CC$, the marking being of norm $\ge 2\pi$.

Thanks to the reducing argument Theorem \ref{thm:reducingargument1}, we can just work in $\mathcal{C}(\CC)$.
Let $G$ be a parabolic group in $\CS_2$ isomorphic to $\ZZ^2$. By symmetry, we can assume without loss of generality that its fixed point is $(0,1)\in \Dp$ (see Subsection \ref{cp1}).
Assume there is a path $t\in [0,1] \mapsto G_t\in \mathcal{C}(\CC)$ such that $G_t\in X$ for every $t\in [0,t)$, and with $G_1 = G$. We would like to find a contradiction.

For all $t$, we have a pair of fixed points $(\zeta_1(t), \xi_1(t)), (\zeta_2(t), \xi_2(t))\in \Dp$, converging each to $(0,1)$.
Set $\delta(t) = 2i\pi R(t) e^{i \omega(t)} =  \dfrac{1}{\zeta_2(t)\xi_1(t) - \zeta_1(t)\xi_2(t)}$, and $\Gamma_t=  R(t) e^{i \omega(t)}  \Log \Xi(t)\subset \CC$, $\Xi_t$ being the subgroup of 
rotation/complex translation quantities of elements of $G_t$.

 Then by definition, each $\Gamma_t$ contains the point $\delta(t)$; since $G_1=G$ is parabolic and all the maps defined insofar are continuous on the parameter $t$, $\delta(t) \to \infty$ as $t \to 1$.
Also, the continuous path $ t \mapsto \Gamma_t$ verifies that $\Gamma_1 = \Gamma$ is the subgroup of translation quantities of $G$, for the normalization defined for elements in $\PS_2$ with fixed point inside $\Dp$ (see Subsection \ref{subsec:parabolic}).
The space of lattices being an open subset of $\mathcal{C}(\CC)$, we may assume that $\Gamma_t$ is a lattice for any $t \in [0,1]$. 
 
 Let $g_1, g_2$ be generators of $G$. Then we can define generators $g_1(t), g_2(t)$ of $\Gamma_t$ so that $g_1(t) \to g_1$ and $g_2(t) \to g_2$ as $t \to 1$. 
 For small $\epsilon > 0$, let $N_{\epsilon}$ be a neighborhood of $\Gamma$ in $\CCC(\CC)$ such that  the following holds: $\Gamma_t$ lies in $N_{\epsilon}$ if and only if 
 $|g_1 - g_1(t)| < \epsilon$ and $|g_2 - g_2(t)| < \epsilon$.  
 For each $t$, there are some integers $k_1(t), k_2(t)$ such that  $k_1(t)g_1(t) + k_2(t)g_2(t) = \delta(t)$.
But since $t\mapsto k_i(t)$ is continuous and $[0,1]$ is connected, $k_1, k_2$ are constant maps. This contradicts the fact that $\delta$ blows up to $\infty$ when approaching 1. 

Hence such a continuous path $t\mapsto G_t$ cannot exist, and we are done for the proof of Theorem \ref{thm:DichotomyParabolics} in the first case.

Moreover, the present proof shows that each choice of $k_1$, $k_2$ corresponds to a choice of connected component in a neighborhood of the lattice $\Gamma$.
Such a choice also sums up to a choice of $\delta$. In the view of geometric interpretations of $R$ and $\omega$ given in Subsection \ref{subsec:geometricviewRomega}, we would like to visualize the possible choices of $1/\delta$ instead.
Figures \ref{fig:nearz2} and \ref{fig:antipodalgluing} below describe the situation. 

 \begin{figure}[ht]
\begin{center}
\includegraphics[scale=0.36]{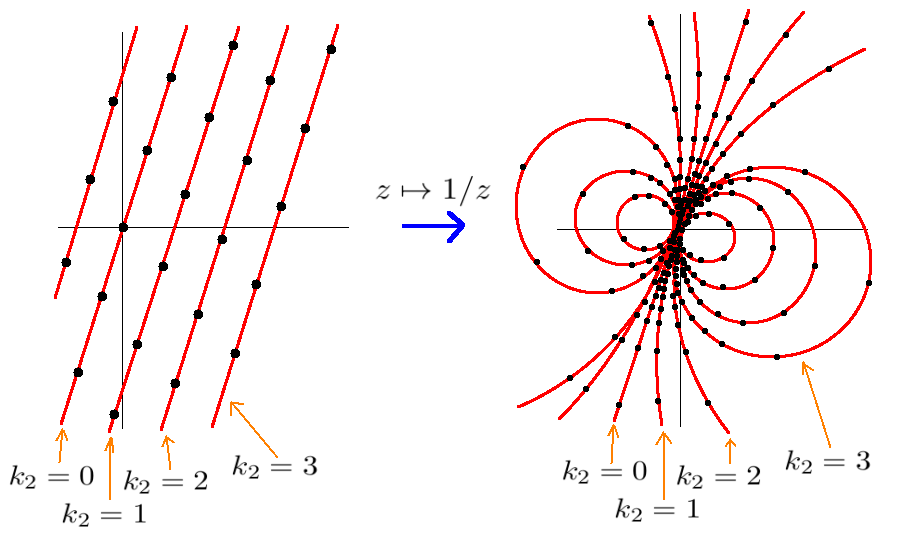}
\end{center}
\caption[How one can see the neighborhood of $\Gamma$ of type $\ZZ^2$ (1) ]{Each dot corresponds to a choice of $k_1, k_2$, or equivallently to a possible location of  $\delta$.
In the left picture, the red lines represent different choices for $k_2$ and on each line, there are dots corresponding to choices for $k_1$.
 On the right, we can see the possible locations of $1/\delta$.}
\label{fig:nearz2}
\end{figure}

In Figure \ref{fig:nearz2}, one should imagine that around each dot, there is a little ball of dimension 6 corresponding to a choice for a first fixed point $f_1$ close to 0, and choices for two generators $g_1'$, $g_2'$ $\epsilon$-close to $g_1$ and $g_2$ respectively. 
The second fixed point $f_2$ is not arbitrarily, since $\delta/2i\pi$ gives the  position of $f_2$ relative to $f_1$.

These quantities describe every non-parabolic group in $\CS_2$ that is close to $G$ exactly twice, because $f_1$, $g_1'$, $g_2'$, $k_1$, $k_2$ and $f_2$, $g_1'$, $g_2'$, $-k_1$, $-k_2$ represent the same group.
Therefore, to get a correct picture of a neighborhood of $G$, we have to forget about the circles on one side of the second picture in Figure \ref{fig:nearz2}, and to glue the remaining line with itself, as in Figure \ref{fig:antipodalgluing}.
 
\begin{figure}[ht]
\begin{center}
\includegraphics[scale=0.36]{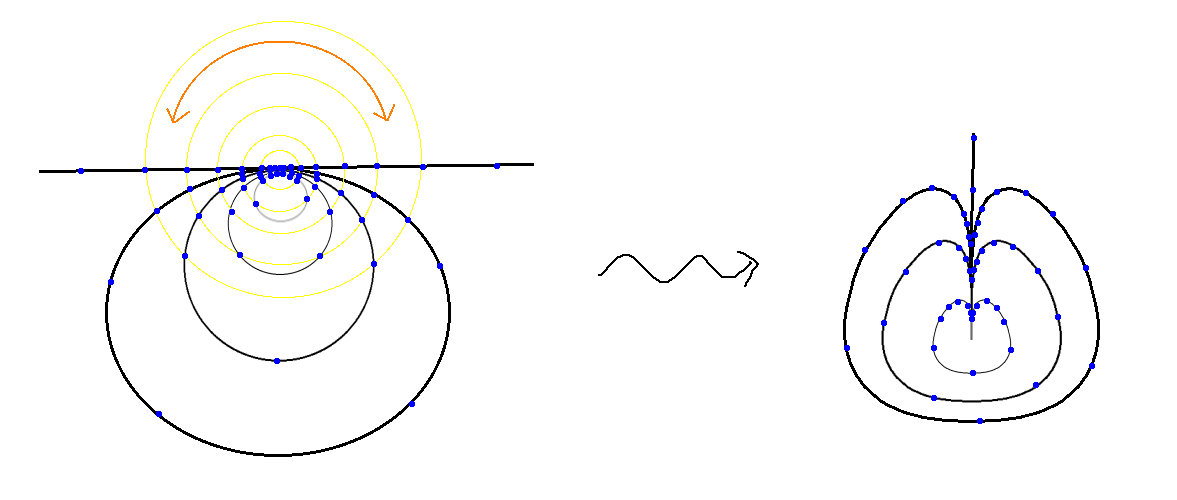}
\end{center}
\caption[How one can see the neighborhood of $\Gamma$ of type $\ZZ^2$ (2) ]{To get a correct model for a neighborhood of $P$, take the pictures of Figure \ref{fig:nearz2} and perform the gluing indicated. Also, replace the black dots by 6-dimensionnal balls.}
\label{fig:antipodalgluing}
\end{figure}

 \subsection{Non-spiraling cases} 
 \label{subsec:nonspiraling}
There are several subcases that we would like to investigate now.
 \begin{itemize}
 \item[(1)] $G$ is parabolic, isomorphic to $\ZZ$,
 \item[(2)] $G$is parabolic, isomorphic to $\RR \times \ZZ$,
 \item[(3)] $G$ is parabolic, isomorphic to $\RR$,
 \item[(4)] $G$ is parabolic, isomorphic to $\CC$, 
 \item[(5)] $G$ is the trivial subgroup $\{1\}$. 
 \end{itemize} 
  
  (1) Let $G$ be a parabolic group in $\CS_2$ isomorphic to $\ZZ$. We can assume without loss of generality that its fixed point is $(0,1)\in \Dp$ and that its associated subgroup $\Gamma$ is $\ZZ\subset \CC$.
    For small $\epsilon > 0$, let $N_{\epsilon}$ be the neighborhood of $\Gamma$ in $\CCC(\CC)$ consisting of the subgroups $g_1 \ZZ$ and of the lattices $\langle g_1, g_2\rangle$ with $|1-g_1|<\epsilon$ and $g_2\in \{|\RE(z)|<1\text{ and }\IM(z)>1/\epsilon\}$.
    Take $G_0$ in a little neighborhood $U$ around $G$ in $\CS_2$ for which every element has associated subgroup in $N_\epsilon$.
 Write $\Gamma_0$ for the associated subgroup of $G_0$; it equals either $g_1(0) \ZZ$ or $\langle g_1(0), g_2(0)\rangle$.
An argument as above would show easily that if $\delta(0)=k_1 g_1$ then there are no continuous paths $t\mapsto G_t\in U$ such that $G_1=G$.
Moreover, each different choice for $k_1>0$ corresponds to a different connected component for $U$.

Now if $\delta(0)=k_1 g_1+k_2 g_2$ with $k_2>0$, then the two paths 
\[
 t\mapsto \langle (1-t)g_1+t, (1-t)g_2+2it/\epsilon \rangle
\]
and 
\[
 t\mapsto \langle (1-t)g_1+t, g_2+\frac{i}{1-t} \rangle
\]
show that each different choice for $k_2>0$ (no matter what $k_1 $ is) corresponds to at least one connected component for $U$, that intersects $\PS_2$ non-trivially.
Since $k_2(t)$ is easily seen to be locally constant, it must be constant.

Therefore the whole picture in that case is similar to the right one in Figure \ref{fig:antipodalgluing}, but with each circle bounding a disk (corresponding to the region where $k_2 g_2$ can be), and with the interiors of these disks disjoint. 
Hence we obtain something looking like a  lollipop with infinitely many layers! See Figure \ref{fig:nearz} below.

\begin{figure}[ht]
\begin{center}
\includegraphics[scale=0.30]{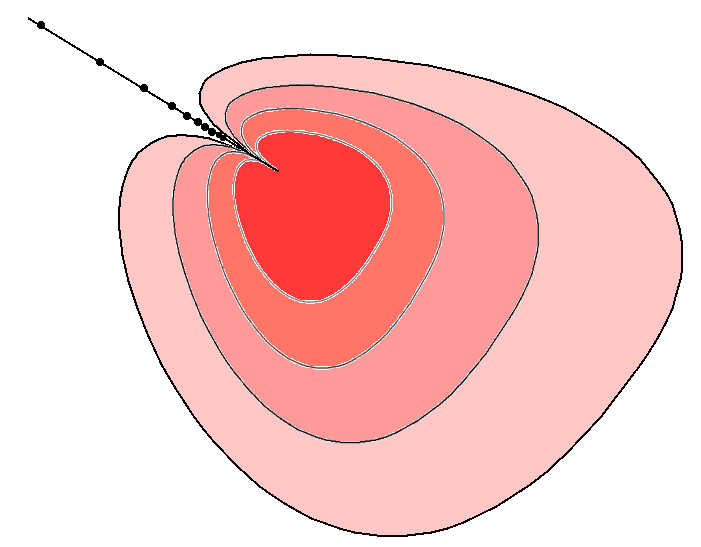}
\end{center}
\caption[How one can see the neighborhood of $\Gamma$ of type $\ZZ$]{ Here each colored region is a disk with one cusp point on the boundary which points inward. Those cusped disks have disjoint interiors and are attached to each other at the cusp
point. As in Figure \ref{fig:nearz2}, this should be thickenned into a 6-dimensionnal object. In particular, the cusp point should actually represent a whole neighborhood of $G$ in $\PS_2$.}
\label{fig:nearz}
\end{figure}  
  
\begin{rmk}
 Note that each layer corresponding to a chosen $k_2>0$ is of dimension 6, and accumulates onto $U\cap \PS_2$, which is of dimension 6 also.
This accumulation looks much like the accordion example provided above, where here the flat direction is the vertical one for $g_2\to \infty$ (used in the proof), and the pleated one is the horizontal one, as we see by considering
\[
 t\mapsto \langle (1-t)g_1+t, g_2+\frac{1}{1-t}, \rangle
\]
which leads to a path $\delta$ with $\delta(t)\to \infty$ but approaching every parabolic point having associated subgroup $\langle 1, \IM(g_2)+u \rangle$, $u$ in $\RR$.
\end{rmk}

  (2) 
Let $G$ be a parabolic group in $\CS_2$ isomorphic to $\ZZ\times \RR$. We can assume without loss of generality that its fixed point is $(0,1)\in \Dp$ and that its associated subgroup $\Gamma$ is $i\ZZ+\RR$.
    For small $\epsilon > 0$, let $N_{\epsilon}$ be a neighborhood of $\Gamma$ in $\CCC(\CC)$ consisting of the subgroups $\langle g_1, g_2\rangle$ with $|g_1|<\epsilon$, $|\Arg(g_1)|<\epsilon$ and $|i-g_2|<\epsilon$, and of subgroups of $\CC$ isomorphic to $\ZZ\times\RR$ and close enough to $\Gamma$.
Now, as above, let $U$ be a neighborhood of $G$ for which every element has associated subgroup in $N_\epsilon$.

Discussions as before show that lattice subgroups in $N_\epsilon$ with a choice $\delta(0)= k_1 g_1 + k_2 g_2$, $k_2>0$ each correspond to one connected component of $U$ that contains all $U\cap \PS_2$. 
The novelty in the case $k_2=0$ is that $g_1$ can be made to converge to 0. In the process, $\delta(t)=k_1g_1(t)$ has to get close to 0 also, thus $G_t$ needs to exit $U$ at some point.
Thus a local model should look like in Figures \ref{fig:nbhdofRZ1} and \ref{fig:nbhdofRZ2}.

\begin{figure}[ht]
\begin{center}
\includegraphics[scale=0.37]{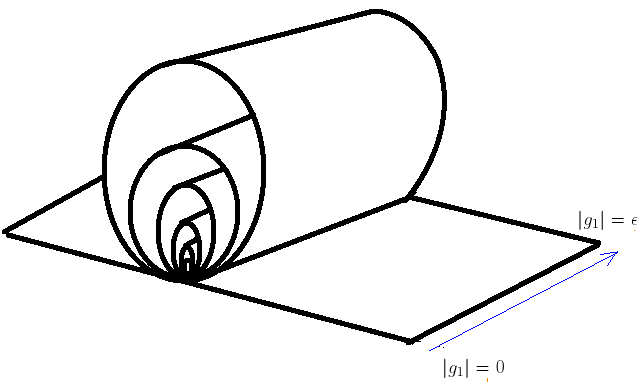}
\includegraphics[scale=0.37]{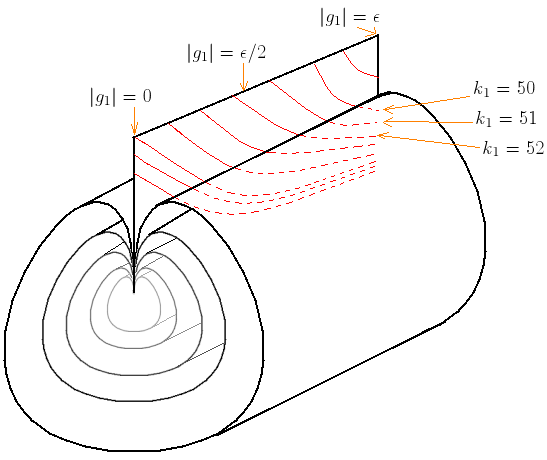}
\end{center}
\caption[How one can see the neighborhood of $\Gamma$ of type $\ZZ \times \RR$ ]{ Left: local model around a parabolic group isomorphic to $\ZZ\times \RR$, viewed as the drawing from Figure \ref{fig:nearz2} cross a short interval representing the possible values of $|g_1|$. The flat part where the cylinders are attached is not filled: it is laminated; this lamination is described in more details in Figure \ref{fig:nbhdofRZ2}. Right: same drawing, with the identification explained in Figure \ref{fig:antipodalgluing}.}
\label{fig:nbhdofRZ1}
\end{figure}  

\begin{figure}[ht]
\begin{center}
\includegraphics[scale=0.43]{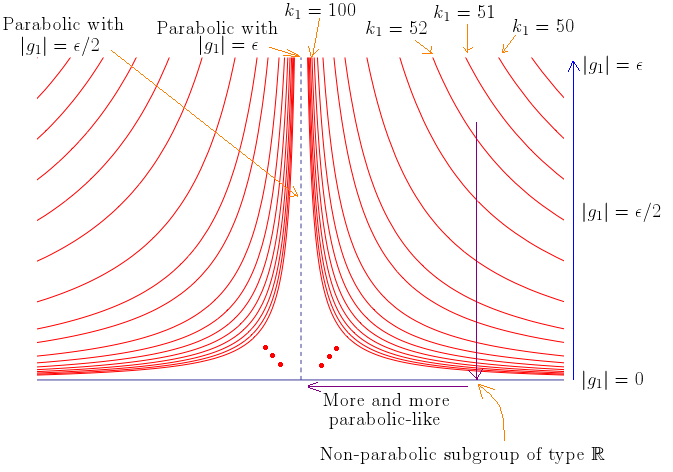}
\end{center}
\caption[How one can see the neighborhood of $\Gamma$ of type $\ZZ \times \RR$ ]{ 
Here we see the flat part of Figure \ref{fig:nbhdofRZ1}.}
\label{fig:nbhdofRZ2}
\end{figure}  

  (3) The case where $G$ is isomorphic to $\RR$ is similar to the case (2) but slightly more complicated, since now it is possible to approach $\Gamma$ by groups of type $\ZZ$.  
  The difference between this case and the previous case is similar to the one between Figure \ref{fig:antipodalgluing} and Figure \ref{fig:nearz}. 
  Namely, one should fill in the cylinders to make them into solid cylinders (but disjoint). See Figure \ref{fig:nbhdofR}.
  The repartition of the groups on the flat part is similar.

\begin{figure}[ht] 
\begin{center}
\includegraphics[scale=0.37]{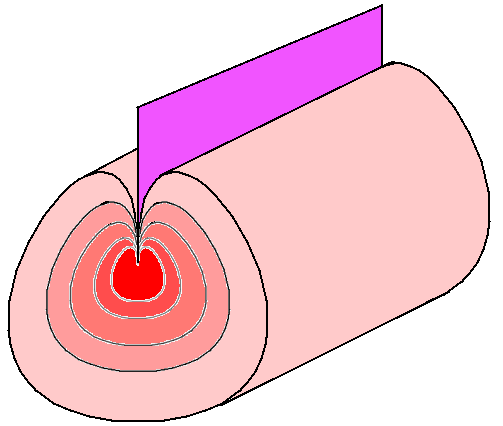}
\end{center}
\caption[How one can see the neighborhood of $\Gamma$ of type $\RR$ ]{Solid lollipop: a local model around a parabolic group isomorphic to $\RR$.}
 \label{fig:nbhdofR}
\end{figure}

The cases (4) is somewhat wilder. Indeed, it concerns subgroups of $\CS_2$ isomorphic to $\RR$; these include in particular copies of the neighborhood of the wedge point of the $D$-bouquet in the Chabauty space of $\CC^*$
We will only attempt to show that this case possesses arbitrary small neighborhoods $U$ such that both $U$ and $X\cap U$ are arcwise connected; this statement somewhat represent an ultimate non-spiraling behavior.

  

  (4) Any parabolic group isomorphic to $\CC$ possesses arbitrary small neighborhoods $U$ such that both $U$ and $X\cap U$ are arcwise connected.
To show that, notice that it is sufficient to prove that given any subgroup $\Gamma_0$, with marked point $\delta\neq 0$, close enough to the total subgroup $\Gamma_1=\CC$ (with marked point $\delta$), we can find a path $t\mapsto \Gamma_t$ such that $t\mapsto d(\Gamma_t, \Gamma_1)$ is decreasing (with $d$ the Chabauty metric on $\CCC(\CC)$), and verifying $\delta\in \Gamma_t$ for all $t$.

First, consider $\Gamma_0$ to be a subgroup of $\CC$ of type $\RR \times \ZZ$. 
Define the spacing $\varphi$ to be the distance between two consecutive copies of $\RR$ in $\Gamma_0$.

If $\delta$ is on the copy $L_1$ of $\RR$ in $\Gamma_0$ passing through $0$ 
then we can choose $t\mapsto \Gamma_t$ with $G_t$ the subgroup of $\CC$ isomorphic to $\RR \times \ZZ$,
containing $L_1$ and with spacing $(1-t)\varphi$.

Else, let $L$ be the line passing though $0$ and $\delta$; in this case, performing a continuous rotation of $\Gamma_0$ around 0 such that the angle between $L$ and the line through 0 in $\Gamma_t$ continuously decreases to 0, we again obtain that $\Gamma_t$ converges to $\CC$ without changing
$\delta$. In both cases, we do not have to move the fixed points, and the subgroups $G_t$ verify that $t\mapsto d(\Gamma_t, \CC)$ is decreasing.

Now consider $\Gamma_0$ to be a lattice. Let $g_1 \in \Gamma_0$ be a primitive
element so that $\delta = k g_1$ for some $k \in \NN$. Choose $g_2 \in
\Gamma_0$ so that $g_1, g_2$ generate $\Gamma_0$.
We define a map $f: \CC\times [0,1] \to \CC$ by 
\[f: (xg_1 + yg_2,t) \mapsto xg_1 + (1-t)yg_2.\]
Then obviously $t\mapsto\Gamma_t=f(\Gamma_0 \times \{t\})$
continuously changes $G_0$ into a subgroup of type $\RR \times \ZZ$, 
without changing $\delta$ and with the required property on $d(\Gamma_t, \CC)$.
Thus we can now apply the previous argument to
continuously deform $\Gamma_1$ into $\CC$ in the required way.
This completes the proof.

\mbox{}

  (5) $\{Id\}$ has arbitrary small neighborhoods $U$ such that $U$ is arcwise connected, but $X\cap U$ in not arcwise connected. 

Let's prove the first part of the claim. 
Take $G_0$ to be a subgroup of $\CS_2$ which is of type either $\ZZ$ or a lattice.
Move the fixed point(s) of $G_0$ continuously until one of the two (resp. the only fixed point, in the case where $G_0$ is parabolic) is 0;
note, in the non-parabolic case, that we can make $\delta$ to remain unchanged.

Now, if $G_0$ is of type $\ZZ$, notice that we can always choose a convex path as above from $G_0$ to $G_0+i/\epsilon G_0$ (without moving $\delta$ in the non-parabolic case);
thus, let us suppose that $G_0$ is isomorphic to $Z^2$.

Let $g_1 \in \Gamma_0$ be a primitive element so that $\delta = k g_1$ for some $k \in \NN$. Choose $g_2 \in
\Gamma$ so that $g_1, g_2$ generate $\Gamma_0$.
We define a map $f: \CC\times [0,1] \to \CC$ by 
\[f: (xg_1 + yg_2,t) \mapsto \frac{1}{1-t}xg_1 + yg_2.\]
Then obviously $t\mapsto\Gamma_t=f(\Gamma_0 \times \{t\})$
continuously changes $G_0$ into a \textit{parabolic} subgroup of type $ \ZZ$, since $\delta\to\infty$ in the case where $G_0$ was not already parabolic.

Finally, the family of elements in $\PS_1$ fixing 0 and staying close enough to $\{\Id\}$ in $\CS_2$ is obviously connected and touches $\{\Id\}$. 
This completes the proof of the first claim.

For the second part of the claim, consider points $\Gamma_1$, $\Gamma_2$ in $X \cap U$. Inside each $\Gamma_i$, 
we have $\delta_i = k_i b_i$ with integers $k_i$ and primitive elements $b_i$. 
Call this $k_i$ the multiplicity of $\delta_i$ in $\Gamma_i$. 
Then one can find a continuous path connecting $\Gamma_1$ and $\Gamma_2$ and contained inside $X \cap U$ if and only if $k_1 = k_2$.  
One can show this easily by modifying the argument given above slightly, the only difference here being that $\delta_i$ can not be let to go to $\infty$. 
Hence we have one connected component of $X \cap U$ for each choice of the multiplicity of $\delta$. 

\pagebreak

\section{Summary statement}
\label{sec:future}
The following theorem collects and summarizes all previous results in this paper.
\begin{thm}
The space $X$ of non-trivial non-parabolic closed abelian subgroups of $\PC$ is homeomorphic to 
\[
\Theta \times (\CCC(\CC^\ast)\setminus \{1\})
\]
where $\Theta\cong \CPP \setminus \CP1$ is the space of pairs of points of a 2-sphere (see Subsection \ref{subsec:Theta}) and $\CCC(\CC^\ast)$ is the Chabauty space of $\CCC(\CC^*)$ described in Section \ref{sec:cstar}.
See Proposition \ref{prop:nonpar}.

Moreover, the space of non-trivial discrete cyclic subgroups of $\PC$ generated by an elliptic generator is homeomorphic to 
\[
\Theta \times \NN_{\ge 2}
\] 
and the space of non-trivial discrete cyclic subgroups of $\PC$ generated by a hyperbolic generator is homeomorphic to 
\[
\Theta \times (\CC\setminus \overline{\DD})
\] 
See Proposition \ref{prop:nonpar2}.

The boundary $\PS_1$ of the space of cyclic parabolic elements in $\CS_1\subset \CS_2$ is
the one-point compactification of a 4-twist $\SOO$-bundle of $\overline{\DD}^\ast$ over $S^2$ (see Corollary \ref{cor:p1}). 
It lies inside the space $\PS_2$ of parabolic closed abelian subgroups of $\PC$; $\PS_2$ is homeomorphic to the one-point compactification of $S^2\times\RR^4$ (see Corollary \ref{cor:p2}).

The way $X$ is attached to $\PS_2$ is explained in Section \ref{sec:redlem}, in particular Theorem \ref{thm:reducingargument1}.
See also Subsection \ref{subsec:cylinders} for a geometric interpretation.
How this attachment is performed is similar to the bending described in \cite{BC1}, see Theorem \ref{thm:mainthmofpsl2r}.

The geometric limit of a sequence of cyclic subgroups $H_n$ of $\PC$ with translation quantities $a_n$, converging to a parabolic subgroup, is intimately connected to the expansions in continued fractions of $\theta_n=\Arg(a_n)$; see Section \ref{sec:bcstar}, in particular Subsection \ref{subsec:lastcase}.

Moreover, local models for neighborhoods in $\CS_2$ of parabolic groups $G\in \CS_2$ depend only on the isomorphy type of $G$, as a group.
In the generic case where $G$ is a isomorphic to $\ZZ^2$, $X$ accumulates towards $G$ in a spiraling way.
When $G$ is not, $X$ accumulates toward $G$ in a non-spiraling way (see Section \ref{sec:pieces}).  
\end{thm}

%
%
 An interesting direction in generalizing this work would be to study the case of elementary subgroups of $\PC$.
 Since those are precisely the subgroups of $\PC$ with finite index abelian groups, it seems reasonable to expect that the Chabauty space of
 elementary groups is not too much more complicated.
 
 Also, along the way of Subsection \ref{subsec:lastcase} we discovered that it was possible to relate some aspects of geometric limits to the continued fraction of some quantity (namely $\theta_n$) by reading its expansion, first backwards starting from some index $j+1$, then forwards starting at the index $j+2$. 
We would be really interested in finding other occurences of these relations in other parts of Mathematics.


\pagebreak

\bibliographystyle{alpha}
\bibliography{biblio}

\end{document}